\documentclass[a4paper]{article}

\usepackage{amsthm}
\usepackage{amsfonts}
\usepackage{graphicx}
\usepackage{epstopdf}
\usepackage{algorithm}
\usepackage{algorithmic}

\newcommand{\R}{\mathbb{R}}

\usepackage{amsmath,amssymb}
\DeclareMathOperator{\E}{\mathbb{E}}
\DeclareMathOperator{\Tr}{Tr}

\DeclareMathOperator*{\argmax}{arg\,max}

\usepackage[font=footnotesize,labelfont=bf]{caption}
\usepackage[font=footnotesize,labelfont=bf]{subcaption}

\usepackage{cite}
\usepackage{adjustbox}

\usepackage{mathtools}
\DeclarePairedDelimiterX{\norm}[1]{\lVert}{\rVert}{#1}
\DeclarePairedDelimiterX{\inp}[2]{\langle}{\rangle}{#1, #2}

\newtheorem{remark}{Remark}
\newtheorem{theorem}{Theorem}

\newtheorem{lemma}[theorem]{Lemma}

\title{Accelerated Sampling Kaczmarz Motzkin Algorithm for the Linear Feasibility Problem}

\author{Md Sarowar Morshed \thanks{Department of Mechanical $\&$ Industrial Engineering, Northeastern University, Boston, MA 02115, USA }
\and Md Saiful Islam \footnotemark[1]
\and Md Noor-E-Alam \thanks{Corresponding Author: mnalam@neu.edu} \footnotemark[1]}

\begin{document}
\maketitle

\section*{Abstract}
The \textit{Sampling Kaczmarz Motzkin} (SKM) algorithm is a generalized method for solving large-scale linear systems of inequalities. Having its root in the relaxation method of Agmon, Schoenberg, and Motzkin and the randomized Kaczmarz method, SKM outperforms the state-of-the-art methods in solving large-scale \textit{Linear Feasibility} (LF) problems. Motivated by SKM's success, in this work, we propose an \textit{Accelerated Sampling Kaczmarz Motzkin} (ASKM) algorithm which achieves better convergence compared to the standard SKM algorithm on ill-conditioned problems. We provide a thorough convergence analysis for the proposed accelerated algorithm and validate the results with various numerical experiments. We compare the performance and effectiveness of ASKM algorithm with SKM, \textit{Interior Point Method} (IPM) and \textit{Active Set Method} (ASM) on randomly generated instances as well as Netlib LPs. In most of the test instances, the proposed ASKM algorithm outperforms the other state-of-the-art methods.\\

\textit{\textbf{Keywords:}} Kaczmarz Method,  Nesterov's Acceleration,  Motzkin Method,  Sampling Kaczmarz Motzkin Algorithm

\textit{\textbf{MSC 2010:}} 90C05,  65F10, 90C25, 15A39, 68W20

\maketitle

%


\section{Introduction}
\label{sec:intro}
We consider the problem of solving large-scale systems of linear inequalities:
\begin{align}
\label{eq:1}
Ax \leq b, \ \ b \in \R^m, \ A \in \R^{m\times n}
\end{align}
Since, iterative methods are usually better suited to problems with large number of constraints compared to the number of variables, we confine the scope of this work to the $m \gg n$ regime. We denote the rows of matrix $A$ by $a_i^T$ for $i = 1,2,..,m$. In addition, we make the following assumptions: (1) the system is consistent, (2) matrix $A$ has no zero rows and (3) the rows of $A$ are normalized (i.e. $\|a_i\| =1$). It is worth noting that the last assumption is not a significantly important requirement for algorithmic efficiency, but it helps in the convergence analysis.

 While most classical iterative methods are deterministic, recent works \cite{strohmer:2008,lewis:2010,needell:2010,drineas:2011,zouzias:2013,lee:2013,ma:2015,qu:2016} suggest that randomization can play a huge role in the design of efficient algorithms for solving LF problems and randomized algorithms often perform better than existing deterministic methods. As shown in \cite{haddock:2017}, randomized iterative methods can outperform state-of-the-art methods (i.e., IPM, ASM) for large-scale LF. In the field of large-scale optimization, mainly IPMs, there is a growing interest in approximate Newton-type methods (\cite{dembo:1982,stanley:1996,bellavia:1998,xin:2010,wang:2013,kaifeng:2012,jacek:2013}) which use fast sub-schemes for calculating approximate solutions of large-scale \textit{Linear System} (LS).

The Kaczmarz method for solving LS, discovered in 1937 \cite{kaczmarz:1937}, remained unnoticed to the western research community until the early 1980s, when it found an important application in the area of \textit{Algebraic Reconstruction Techniques} (ART) for image reconstruction \cite{herman:2009}. Since then it has been used for several other areas like digital signal processing, computer tomography, and belongs to a general category of methods including row-action, component solution, cyclic projection, and successive projection methods (see \cite{censor:1981}). It gained immense popularity in the research community after the convergence analysis done in 2009 for the randomized version \cite{strohmer:2008}. The convergence analysis of Strohmer \cite{strohmer:2008} encouraged numerous extensions and generalizations of the randomized Kaczmarz method (see \cite{lewis:2010,needell:2010,zouzias:2013,lee:2013,ma:2015,wright:2016}, for instance when we replace the equality constraints with inequality constraints we get a variant of the original problem.

Motzkin's relaxation method is a variation of the Kaczmarz method which was introduced in the early 1950s \cite{agamon,motzkin} for solving systems of linear inequalities. Since then, it has been rediscovered several times. For instance, the famous perceptron algorithm in machine learning \cite{rosenblatt,ramdas:2014,ramdas:2016} can be thought of as a member of this family of methods. Additionally, the relaxation method has been referred to as the Kaczmarz method with the ``most violated constraint control" or the ``maximal-residual control" \cite{censor:1981,nutini:2016,petra:2016}. The rate of convergence of Motzkin's method depends on step lengths and the so called Hoffman constants \cite{agamon,hoffman}.

Combining both the Kaczmarz and Motzkin method together, the SKM algorithm proposed in \cite{haddock:2017} for solving LF problem given in \eqref{eq:1} requires only $O(n)$ memory storage and it has a linear convergence rate. As shown by the authors, SKM is much more efficient than the state-of-the-art techniques such as IPMs, ASMs, and Kaczmarz Methods. Roughly, the SKM algorithm selects a row out of $\beta$ rows (sampled from $A$) by the maximum violation criterion (i.e. choose the row $i^*$ with $i^* = \argmax_{i \in \tau_k} \{a_i^Tx_k-b_i, 0\}, \ \beta = |\tau_k|$) and then updates the next point as follows:
\begin{align}
\label{eq:2}
  x_{k+1} = x_k- \delta \frac{\left(a_{i^*}^Tx_k -b_{i^*}\right)^+}{\|a_{i^*}\|^2} a_{i^*}  
\end{align}
In equation \eqref{eq:2}, $\delta $ can be $ 0 \leq \delta \leq 2$. Without the loss of generality, we consider $\delta =1$ in this work. The SKM method described in \cite{haddock:2017} overcomes the drawbacks of the individual methods (Kaczmarz, Motzkin) and combines their strengths. By selecting the maximum violated hyperplane from a sample, SKM achieves faster convergence compared to the randomized Kaczmarz method. In addition, per iteration computational cost is cheaper compared to Motzkin's method. Recently, Wright \textit{et. al} \cite{wright:2016} applied the acceleration scheme of Nesterov to the randomized Kaczmarz method. In a different work, Xu \textit{et. al} \cite{xu:2017} investigated the acceleration scheme in the context of the extended randomized Kaczmarz method for least square problems. Moreover, there is a recent work in applying Nesterov scheme in IPMs for solving large-scale linear programming problems \cite{morshed:2018}. The above-mentioned works showed that the introduction of Nesterov's acceleration scheme fasten the convergence of the original method.

In this work, we apply Nesterov's acceleration scheme \cite{nesterov:1983,nesterov:2013,nesterov:2014,nesterov:2005,nesterov:2012} to the generalized SKM algorithm. This can be seen as a generalized accelerated scheme for both randomized Kaczmarz method for solving linear systems as well as linear system of inequalities. It can be noted that with some modification, like the one stated in the work of Lewis \textit{et. al} \cite{lewis:2010}, we can apply this method to linear systems with both equality and inequality constraints. The overarching goal of this paper is to incorporate the ideas of the Kaczmarz method \cite{strohmer:2008,kaczmarz:1937,gordon:1970} for LS and Motzkin's relaxation \cite{haddock:2017} for LF problem and develop an accelerated randomized scheme for solving large-scale LF problem. The paper is organized as follows. The proposed algorithm is discussed in section \ref{sec:alg}, and the convergence analysis of the proposed algorithm is given in section \ref{sec:conv}. Extensive Numerical experiments performed on random and Netlib LP instances are provided in section \ref{sec:num}. And finally the paper is concluded with the conclusion in section \ref{sec:colc}.

\section{ASKM Algorithm}
\label{sec:alg}
\subsection{Notation:}

We follow the standard notation in this work. For example, $\mathbb{R}$ will be used to denote the set of real numbers. Matrix $A$ with $m$ rows and $n$ columns belong to $\mathbb{R}^{m\times n}$, with $A_{ij}$ denoting the real-valued element in row $i$ and column $j$. $A^T$ will be used to denote the transpose of matrix $A$, with $tr(A)$, $det(A)$, and $diag(A)$ denoting the trace, determinant, and diagonal of matrix $A$ respectively. $I_n$ will be used as the $n\times n$ identity matrix.

Furthermore, we use vectors $\mathbf{1}=\left[1~1~\ldots~1\right]^T$ and $e_i$ as the standard $i$-th basis vector. A function $f:X\mapsto Y$ maps its domain, $dom(f)\subseteq X$, into set $Y$. As it is customary, we use $\nabla f$ and $\nabla^2 f$ to represent the gradient and Hessian of $f$. Finally, $\langle x,y\rangle = x^Ty$ denotes the standard inner product and $\|x\| = \sqrt{\langle x, x \rangle}$ as the euclidean ($L_2$) norm. $\lambda_{min}, \lambda_{max}$ are set to be the minimum and maximum nonzero eigenvalues of $A^TA$ respectively. $\|A\|$ is the spectral norm of the matrix $A$ and $\|A\|_F$ denote the Frobenius norm. Moreover, $A^{\dagger}$ is the Moore-Penrose pseduinverse of $A$ and the corresponding compact singular value decomposition of $A \in \mathbb{R}^{m\times n}$ as $A = U \Sigma V^T$, where $U,V$ are unitary matrices with appropriate size and $\Sigma$ is the non-singular and diagonal matrix with singular value on the diagonal. Throughout the paper, we denote $\zeta$ as the condition number of matrix $A$. The notation $\mathcal{P}_{A,b}(x)$ denotes the Euclidean norm projection of $x$ onto the feasible region of $Ax \leq b$. In this section, we review the proposed SKM algorithm in \cite{haddock:2017} and then based on the motivation from the accelerated randomized Kaczmarz algorithm in \cite{wright:2016} and accelerated extended Kaczmarz algorithm in \cite{xu:2017}, we develop ASKM algorithm. 
\begin{algorithm}
\caption{SKM Algorithm: $x_{k+1} = \textbf{SKM}(A,b,x_0,K,\delta, \beta )$}
\label{alg:skm}
\begin{algorithmic}
\STATE{Initialize $k \leftarrow 0$;}
\WHILE{$k \leq K$}
\STATE{Choose a sample of $\beta$ constraints, $\tau_k$}, uniformly at random from the rows of matrix $A$.
\STATE{From these $\beta$ constraints, choose $i^* = \argmax_{i \in \tau_k} \{a_i^Tx_k-b_i, 0\}$;}
\STATE{Update $x_{k+1} = x_k- \delta \frac{\left(a_{i^*}^Tx_k -b_{i^*}\right)^+}{\|a_{i^*}\|^2} a_{i^*}$;}
\STATE{$k \leftarrow k+1$;}
\ENDWHILE
\RETURN $x$
\end{algorithmic}
\end{algorithm}

\begin{algorithm}
\caption{ASKM Algorithm: $x_{k+1} = \textbf{ASKM}(A,b,x_0,K, \beta,  \lambda_{min}, \zeta)$}
\label{alg:askm}
\begin{algorithmic}
\STATE{Choose $\lambda \in [0, \lambda_{min}]$;}
\STATE{Initialize $v_0 \leftarrow x_0, \ \gamma_{-1} \leftarrow 0, \  k \leftarrow 0$;}
\WHILE{$k \leq K$}
\STATE{Choose $\gamma_k$ to be the larger root of
\begin{align}
\label{eq:4}
\gamma_k^2-\frac{\zeta}{m} \gamma_k = \frac{d}{\beta} \left(1- \frac{\lambda \beta }{m} \gamma_k\right) \gamma_{k-1}^2   
\end{align}

}
\STATE{Update $\alpha_k$ and  $ \beta_k$ as follows:
\begin{align}
 \alpha_k & = \frac{\zeta(m- \lambda \beta \gamma_k)}{\gamma_k(m^2-\lambda \zeta \beta)} \label{eq:5} \\
\beta_k & = 1- \frac{\lambda \beta }{m} \gamma_k  \label{eq:6} \ ;  
\end{align}
}
\STATE{Update $y_k = \alpha_k v_k + (1-\alpha_k)x_k$;}
\STATE{Choose a sample of $\beta$ constraints, $\tau_k$, uniformly at random from the rows of matrix $A$. From these $\beta$ constraints, choose $i^* = \argmax_{i \in \tau_k} \{a_i^Tx_k-b_i, 0\}$;}
\STATE{Update 
\begin{align}
x_{k+1} & = y_k- \frac{\left(a_{i^*}^Ty_k -b_{i^*}\right)^+}{\|a_{i^*}\|^2} a_{i^*}; \label{eq:7} \\
v_{k+1} & = \beta_k v_k+ (1-\beta_k)y_k-\gamma_k \frac{\left(a_{i^*}^Ty_k -b_{i^*}\right)^+}{\|a_{i^*}\|^2} a_{i^*} \label{eq:8} ;
\end{align}
}
\STATE{$k \leftarrow k+1$;}
\ENDWHILE
\RETURN $x$
\end{algorithmic}
\end{algorithm}

In the above algorithm, we propose to use the acceleration scheme discovered by Nesterov \cite{nesterov:1983,nesterov:2013,nesterov:2014,nesterov:2005,nesterov:2012} in the SKM algorithm framework to achieve second order convergence rate as compared to the linear rate shown in \cite{haddock:2017}. The ASKM algorithm uses the acceleration procedure \cite{nesterov:2014}, which is more famous in the context of gradient descent algorithm. Note that, Nesterov's acceleration scheme uses two new sequences $\{y_k\}$ and $\{v_k\}$ and update the sequences as follows:
\begin{align}
\label{eq:3}
 & y_k  = \alpha_k v_k + (1-\alpha_k) x_k \nonumber \\
 & x_{k+1}  = y_k - \theta_k \nabla f(y_k) \\
 & v_{k+1} = \beta_k v_k + (1- \beta_k) y_k - \gamma_k \nabla f (y_k) \nonumber
\end{align}
In equation \eqref{eq:3}, $\nabla f$ is the gradient of the given function and $\theta_k$ is the step-size. The main contribution for the above scheme is that it uses appropriate values for the parameters $\alpha_k, \beta_k, \gamma_k$, which in turn yield better convergence in the context of standard gradient descent. Now, using the general setup of Nesterov's scheme \cite{nesterov:2012} for coordinate descent and the idea in \cite{wright:2016}, we developed ASKM algorithm shown above (Algorithm \ref{alg:askm}).

\section{Convergence Analysis}
\label{sec:conv}

In this section, we analyze the convergence of the proposed ASKM algorithm \ref{alg:askm}. Throughout the analysis, we make the assumptions: 1) $\|a_i\| = 1$ for any $i\in m$, which implies $\|A\|_F^2 = m$ and 2) $\mathcal{P}_{A,b}$ is full dimensional. The following convergence result was proven in \cite{haddock:2017} for the SKM algorithm (Algorithm \ref{alg:skm}):
\begin{align}
\label{eq:9}
E\left[\|x_{k+1}-\mathcal{P}_{A,b}(x_{k+1})\|^2\right] & \leq \left(1- \frac{2 \delta -\delta^2}{V_k L^2}\right) \|x_k-\mathcal{P}_{A,b}(x_{k})\|^2 \nonumber \\
& \leq \left(1- \frac{2 \delta -\delta^2}{m L^2}\right)^{k+1} \|x_k-\mathcal{P}_{A,b}(x_{0})\|^2
\end{align}
In the above equation, $L$ is the Hoffman constant and $V_k$ is defined in the proof of Lemma \ref{lem:3}. For the ASKM algorithm (Algorithm \ref{alg:askm}) shown above, we prove a better convergence result as stated in Theorem \ref{th:1} compared to the one in \eqref{eq:9} (we consider the case $\delta =1$). 

\begin{remark}
This framework for convergence in the context of acceleration follows the general idea developed by Nesterov \cite{nesterov:2013} for the \textit{Gradient Descent} method. The proof of Theorem \ref{th:1} follows the generalized sketch developed by Nesterov \cite{nesterov:2012} for proving the convergence result of \textit{Coordinate Descent} method. Due to the similarity of acceleration methods derived in \cite{wright:2016} for the randomized Kaczmarz method and our proposed method, we will use the same standard notation on this subject. In addition to that, the following results generalize results for acceleration in Kaczmarz types methods (i.e. if we select $\beta =1$ and use linear systems, we get the same results shown in \cite{wright:2016}).
\end{remark}

\begin{theorem}
\label{th:1}
The ASKM algorithm defined above with $\lambda \in [0, \lambda_{min}]$ and $\sigma_1 = 1+\frac{\sqrt{\lambda \beta \zeta}}{2m}, \sigma_2 = 1-\frac{\sqrt{\lambda \beta \zeta}}{2m}$, then for all $k \geq 0 $ we have the following:
\begin{align}
& \E \left[\big\|v_{k+1}-x^*\big\|^2_{(A^TA)^{\dagger}}\right] \ \leq \ \frac{4\big\|x_0-x^*\big\|^2_{(A^TA)^{\dagger}}}{\left(\sigma_1^{k+1}+\sigma_2^{k+1}\right)^2} \label{eq:10} \\
& \E \left[\big\|x_{k+1}-x^*\big\|^2\right] \ \leq \ \frac{4 \lambda\big\|x_0-x^*\big\|^2_{(A^TA)^{\dagger}}}{\zeta \left(\sigma_1^{k+1}-\sigma_2^{k+1}\right)^2} \label{eq:11}
\end{align} 
Here, $x^* \in \mathbb{R}^{n}$ is a unique limit point of the ASKM iterates (for the uniqueness of $x^*$ see Lemma 2.2-2.4 in \cite{haddock:2017}), $\zeta$ is the condition number of matrix $A$ and $\beta$ is the sample size of the random sampling process.
\end{theorem}
Before delving into Theorem \ref{th:1}, we start with the proof of some useful lemmas. For the expectation calculation of the random process described in our algorithm, we need a certain setup. Let, $(Ax_k-b)^+_{i_j}$ denote the $(\beta+j)^{th}$ smallest entry of the residual vector (if we order the entries of $(Ax_k-b)^+$
from smallest to largest, $(Ax_k-b)^+_{i_j}$ is the entry in $(\beta+j)^{th}$ position). Now, if we consider the size of all the entries of the residual vector $(Ax_k-b)^+$, we can calculate the probability that a particular entry of the residual vector is selected. In this case, each sample has equal probability of selection (i.e., $\frac{1}{\binom{m}{\beta}}$). Moreover, the size of the residual vector controls the frequency that each entry of the residual vector will be expected to be selected (Algorithm \ref{alg:askm}, sample of constraints selection). For example, if we have only one sample then $\beta^{th}$ smallest entry will be selected and for the case of all samples, $m^{th}$ smallest entry will be selected. Therefore, if we expand the expectation of the residual (with respect to the probabilistic choice of sample constraints, $\tau_j$, of size $\beta$), we get the following:
\begin{align}
\label{eq:12}
\E_{\beta_{i^*}} \left[\big\|\left(a_i^Ty-b_i\right)^{+}_{i^*}\big\|^2\right] = \frac{1}{\binom{m}{\beta}} \sum\limits_{k = 0}^{m-\beta} \binom{\beta-1+k}{\beta-1} \big | \left(a_i^Ty-b_i\right)^{+}_{i_k} \big |^2    
\end{align}
where $\E_{\beta_{i^*}}$ denotes the required expectation in accordance with the above sampling process ($\beta$ is the sample size).

\begin{lemma}
\label{lem:1}
For any $y \in \R^n,$ we have the following:
\begin{align}
& \E_{\beta_{i^*}} \left[\big\|a_{i^*}\left(a_i^Ty-b_i\right)^{+}_{i^*}\big\|^2_{(A^TA)^{\dagger}}\right] \ \leq \ \frac{\beta}{m} \ \| (Ay-b)^{+}\|^2 \label{eq:13}
\end{align}
\end{lemma}

\begin{proof}
Let us define the singular value decomposition of $A$ as $A = U \Sigma V^T,$ where both $U$ and $V$ are unitary matrices of appropriate dimension and $\Sigma$ is a positive diagonal. We can easily show that $(A^TA)^{\dagger} = V \Sigma^{-2}V^T$. Then, with the defined orientation above, we have the following:
\allowdisplaybreaks{
\begin{align*}
\E_{\beta_{i^*}} & \left[\big\|a_{i^*}\left(a_i^Ty-b_i\right)^{+}_{i^*}\big\|^2_{(A^TA)^{\dagger}}\right]  \\
& = \frac{1}{\binom{m}{\beta}} \sum\limits_{k = 0}^{m-\beta} \binom{\beta-1+k}{\beta-1} \big\|a_{i_k}\big\|^2_{(A^TA)^{\dagger}} \big | \left(a_i^Ty-b_i\right)^{+}_{i_k} \big |^2 \\
& \leq \ \frac{\binom{m-1}{\beta-1}}{\binom{m}{\beta}} \sum\limits_{k = 0}^{m-\beta} \big\|a_{i_k}\big\|^2_{(A^TA)^{\dagger}} \big | \left(a_{i_k}^Ty-b_{i_k}\right)^{+} \big |^2 \\
& \leq \ \frac{\beta}{m} \sum\limits_{j = 1}^{m} \big\|a_{j}\big\|^2_{(A^TA)^{\dagger}} \big | \left(a_{j}^Ty-b_{j}\right)^{+} \big |^2 \\
& = \ \frac{\beta}{m} \sum\limits_{j = 1}^{m} \ \inp[\Big]{ (A^TA)^{\dagger} a_{j}\left(a_j^Ty-b_j\right)^{+}}{ a_{j}\left(a_j^Ty-b_j\right)^{+}}  \\
& = \frac{\beta}{m} \Tr\left[ (A^TA)^{\dagger} \sum\limits_{j = 1}^{m} \ a_{j}\{\left(a_j^Ty-b_j\right)^{+}\}^2 a_j^T\right] \\
& = \frac{\beta}{m} \Tr\left[ (A^TA)^{\dagger} A^T D^2 \left[\left(Ay-b\right)^{+} \right] A\right] \\
& = \frac{\beta}{m} \Tr\left[ V \Sigma^{-2}V^T V \Sigma U^T D^2 \left[\left(Ay-b\right)^{+} \right] U \Sigma V^T \right] \\
& = \frac{\beta}{m} \Tr\left[U^T D^2 \left[\left(Ay-b\right)^{+} \right] U \right] \\
& = \frac{\beta}{m} \big \|D \left[\left(Ay-b\right)^{+} \right] U \big \|^2_{F} \\
& = \frac{\beta}{m} \sum\limits_{j =1}^{m} \big | \left(a_{j}^Ty-b_{j}\right)^{+} \big |^2 \|U_j\|^2_2 \\
& \leq \ \frac{\beta}{m} \sum\limits_{j =1}^{m} \big | \left(a_{j}^Ty-b_{j}\right)^{+} \big |^2 = \ \frac{\beta}{m} \| (Ay-b)^{+}\|^2
\end{align*}}
This proves Lemma \ref{lem:1}.
\end{proof}

\begin{lemma}
\label{lem:2}
For any $y \in \R^n,$ we have the following:
\begin{align}
& \E_{\beta_{i^*}} \left[\big\|a_{i^*}\left(a_i^Ty-b_i\right)^{+}_{i^*}\big\|^2\right] \ \leq \ \frac{\beta}{m} \ \| (Ay-b)^{+}\|^2 \label{eq:14}
\end{align}
\end{lemma}

\begin{proof} With the expression of expectation defined in \eqref{eq:12} we have,
\begin{align*}
\E_{\beta_{i^*}} & \left[\big\|a_{i^*}\left(a_i^Ty-b_i\right)^{+}_{i^*}\big\|^2\right]  \\
& = \frac{1}{\binom{m}{\beta}} \sum\limits_{k = 0}^{m-\beta} \binom{\beta-1+k}{\beta-1} \big\|a_{i_k}\big\|^2 \big | \left(a_i^Ty-b_i\right)^{+}_{i_k} \big |^2 \\
& \leq \ \frac{\binom{m-1}{\beta-1}}{\binom{m}{\beta}} \sum\limits_{k = 0}^{m-\beta}  \big | \left(a_{i_k}^Ty-b_{i_k}\right)^{+} \big |^2 \\
& \leq \ \frac{\beta}{m} \sum\limits_{j =1}^{m} \big | \left(a_{j}^Ty-b_{j}\right)^{+} \big |^2 = \ \frac{\beta}{m} \| (Ay-b)^{+}\|^2
\end{align*}
This proves the Lemma \ref{lem:2}.
\end{proof}

\begin{lemma}
\label{lem:3}
For any $y \in \R^n$ and $x^*$ that satisfies $Ax^*\leq b,$ we have the following:
\begin{align}
&  \frac{1}{m} \ \| (Ay-b)^{+}\|^2 \leq \ \|y-x^*\|^2-\E_{\beta_{i^*}} \left[\big\|\mathcal{P}_{a_{i^*}, b_{i^*}}(y)-x^*\big\|^2\right] \label{eq:15}
\end{align}
\end{lemma}

\begin{proof}
Let us define $\mathcal{P}$ as the projection operator onto the feasible region $P = \{x \in \R^n \ | \ Ax \leq b\}$, and denote $s_k$ as the number of zero entries in the residual $(Ax_k - b)^+$, which also corresponds to number of satisfied constraints. We also define $V_j = \max\{m - s_j, m - \beta + 1\}$. Now, from the update formula shown in Algorithm \ref{alg:askm}, we know that $x_{k+1} = y_k - \frac{(A_{\tau_k}y_k - b_{\tau_k})^+_{i^*}}{\|a_{i^*}\|^2} a_{i^*}$; where,
\begin{align}
i^* = \argmax_{i \in \tau_k} \{a_i^Tx_k-b_i, 0\} \ = \ \argmax_{i \in \tau_k} (A_{\tau_k}x_k-b_i)^+_i \label{eq:16}
\end{align}
Then we have,
\begin{align}
\|x_{k+1}- P\|^2 & = \|x_{k+1}- \mathcal{P}(x_{k+1})\|^2 \leq \|x_{k+1}- \mathcal{P}(y_{k})\|^2 \nonumber \\
& = \big \|y_k - \frac{(A_{\tau_k}y_k - b_{\tau_k})^+_{i^*}}{\|A_{\tau_k}\|^2} a_{i^*}- \mathcal{P}(y_{k})\big \|^2 \nonumber \\
& = \|y_k - \mathcal{P}(y_{k}) \|^2 + [(A_{\tau_k}y_k - b_{\tau_k})^+_{i^*}]^2 \nonumber \\ 
& \quad \quad \quad \quad \quad \quad  \quad \quad \quad \quad - 2 (A_{\tau_k}y_k - b_{\tau_k})^+_{i^*} a_{i^*}^T(y_k - \mathcal{P}(y_{k})) \nonumber \\
& \leq \ \|y_k- \mathcal{P}\|^2 +  \left[(A_{\tau_k}y_k - b_{\tau_k})^+_{i^*}\right]^2  - 2 (A_{\tau_k}y_k - b_{\tau_k})^+_{i^*} (a_{i^*}^Ty_k - b_{i^*}) \nonumber \\
& = \|y_k- \mathcal{P}\|^2 -  \left[(A_{\tau_k}y_k - b_{\tau_k})^+_{i^*}\right]^2 \nonumber \\ 
& = \|y_k- \mathcal{P}\|^2 -  \|(A_{\tau_k}y_k - b_{\tau_k})^+_{i^*}\|^2_{\infty} \label{eq:17}
\end{align}

Now, taking expectation in both sides of equation \eqref{eq:17}, we have, 
\begin{align*}
\E_{\beta_{i^*}}& \left[\big\|x_{k+1}- \mathcal{P}\big\|^2\right]  =  \|y_k-x^*\|^2 - \E_{\beta_{i^*}}\left[\|(A_{\tau_k}y_k - b_{\tau_k})^+_{i^*}\|^2_{\infty}\right] \\
& = \|y_k-x^*\|^2 - \frac{1}{\binom{m}{\beta}} \sum\limits_{j = 0}^{m-\beta} \binom{\beta-1+j}{j} \left[(Ay_k - b)^+_{i_j}\right]^2 \\
&  \leq \  \|y_k-x^*\|^2 - \frac{1}{m-\beta+1} \sum\limits_{j = 0}^{m-\beta}  \left[(Ay_k - b)^+_{i_j}\right]^2 \\
& \leq \ \|y_k-x^*\|^2 - \frac{1}{m-\beta+1} \min\{\frac{m-\beta+1}{m-s_k}, 1\} \sum\limits_{i = 1}^{m}  |(Ay_k - b)^+_{i}|^2 \\
& \leq \  \|y_k-x^*\|^2 - \frac{1}{m} \| (Ay_k-b)^{+}\|^2  
\end{align*}
The expectation above proves the Lemma \ref{lem:3}.
\end{proof}

\noindent \textbf{Definition:} Let us define a function $f: \R^n \rightarrow \R$ as follow:
\begin{align}
\label{eq:18}
f(x) = \frac{1}{2} \left[\left(Ax-b\right)^+\right]^TA\left[(A^TA)^{\dagger}\right]^2 A^T\left[\left(Ax-b\right)^+\right]    
\end{align}
The gradient of the function is given by:
\begin{align}
\label{eq:19}
\nabla f(x) = (A^TA)^{\dagger} A^T \left(Ax-b\right)^+     
\end{align}

\begin{lemma}
\label{lem:4}
For any $x,y \in \R^n$ and condition number of $A$ matrix $\zeta = \frac{\sigma_{\max}(A)}{\sigma_{\min}(A)} = \frac{\lambda^2_{\max}(A^TA)}{\lambda^2_{\min}(A^TA)}$, we have the following:
\begin{align*}
\inp[\Big]{\nabla f(x)}{y-x} \leq f(y) - f(x) + \frac{\zeta}{2} \|x-y\|^2
\end{align*}
\end{lemma}

\begin{proof}
We first prove that $\nabla f$ is Lipschitz continuous  with the constant $\zeta$. Using the definition given in \eqref{eq:18}, for any $x, y \in \R^n$ we have,
\begin{align*}
  \|\nabla f(x) - \nabla f(y)\| & = \|(A^TA)^{\dagger} A^T \left\{(Ax-b)^+ - (Ay-b)^+ \right\}\| \\
  & \leq \|(A^TA)^{\dagger} A^T\| \|(Ax-b)^+ - (Ay-b)^+\| \\
  & \leq \|A^{\dagger}\| \|A\| \|x-y\| \\
  & = \zeta \|x-y\|
\end{align*}
The above equation shows that $\nabla f$ is Lipschitz continuous  with the constant $\zeta$. Here, we use the common expression $(A^TA)^{\dagger} A^T = (A)^{\dagger}$. Now using Lemma 1.2.3 proven in \cite{nesterov:2014}, as $\nabla f$ is Lipschitz continuous, for any $x, y \in \R^n$ we can write the following:
\begin{align}
\label{eq:new1}
\big|f(y) - f(x)-\inp {\nabla f(x)}{y-x}\big| \leq \ \frac{\zeta}{2} \|x-y\|^2
\end{align}
Now, by simplifying \eqref{eq:new1}, we get the following bound:
\begin{align*}
\inp[\Big]{\nabla f(x)}{y-x} \leq f(y) - f(x) + \frac{\zeta}{2} \|x-y\|^2
\end{align*}
The bound mentioned above proves the Lemma \ref{lem:4}.
\end{proof}

\begin{lemma}
\label{lem:5}
For any $m^2 > \zeta \lambda \beta$ and with the following definitions:
\begin{align*}
\beta_k = 1- \frac{\gamma_k \lambda \beta}{m}, \ \alpha_k = \frac{\zeta(m- \lambda \beta \gamma_k)}{\gamma_k(m^2- \zeta \lambda \beta)}
\end{align*}
both sequences $\{\alpha_k\}, \ \{\beta_k\}$ lies in the interval $[0,1]$ if and only if $\gamma_k$ satisfies the following property:
\begin{align}
\label{eq:21}
 \frac{\zeta}{m} \leq \gamma_k \leq \frac{m}{\lambda \beta}   
\end{align}
\end{lemma}

\begin{proof}
The proof of Lemma \ref{lem:5} is straightforward. If we consider the definitions of the sequences $\{\alpha_k\}, \ \{\beta_k\}$ with the given condition, we find that $\alpha_k, \beta_k \in [0,1]$ implies that the following bound must hold:
\[\frac{\zeta}{m} \leq \gamma_k \leq \frac{m}{\lambda \beta}\]
Conversely, if we assume the bound holds for $\gamma_k$, then we can easily find that it implies the sequences $\{\alpha_k\}$ and $\{\beta_k\}$ lies in the interval $[0,1]$.
\end{proof}

\begin{lemma}
\label{lem:6}
For any $d \geq \beta$, if $\gamma_{k-1} \leq \sqrt{\frac{\zeta }{\lambda d}}$ holds, then $\gamma_k$ satisfies the bound in Lemma \ref{lem:5} and also $\gamma_k$ lies in the interval $[\gamma_{k-1},\sqrt{\frac{\zeta }{\lambda d}}]$.
\end{lemma}

\begin{proof}
Let us define the function $g : \R \rightarrow \R$ as follows:
\begin{align}
\label{eq:22}
 g(\gamma) = \gamma^2 + \frac{\gamma}{m} \left(\lambda d \gamma_{k-1}^2-\zeta \right)- \frac{d}{\beta}\gamma_{k-1}^2   
\end{align}
As we know from the definition, $\gamma_k$ is the largest root of $g(\gamma)$, then it satisfies $g(\gamma_k) = 0$. Now we have,
\begin{align*}
g\left(\frac{\zeta}{m}\right) & = \frac{\zeta^2}{m^2} + \frac{d \zeta}{m^2}\lambda \gamma_{k-1}^2 - \frac{\zeta^2}{m^2} - \frac{d}{\beta}\gamma_{k-1}^2 \\
& = \frac{d \gamma_{k-1}^2}{\beta m^2} \left(\lambda \zeta \beta - m^2\right) \leq 0 \ \ \ \ \left(  m^2 > \zeta \lambda \beta\right)
\end{align*}
Similarly,
\begin{align*}
g\left(\frac{m}{\lambda \beta}\right) & = \frac{m^2}{\lambda^2 \beta^2} - \frac{\zeta}{m} \frac{m}{\lambda \beta} - \frac{d}{\beta} \gamma_{k-1}^2+ \frac{\lambda d}{m} \gamma_{k-1}^2 \frac{m}{\lambda \beta} \\
& = \frac{m^2-\lambda \zeta \beta}{\lambda^2 \beta^2} \geq 0
\end{align*}
Therefore, we can write,
\[\frac{\zeta}{m} \leq \gamma_k \leq \frac{m}{\lambda \beta}\]
This proves the first part of the Lemma. For the second part, notice that, assuming $\beta \leq d$ we have the following:
\begin{align*}
g(\gamma_{k-1}) & = \gamma_{k-1}^2 +\gamma_{k-1} \frac{1}{m}\left(\lambda d \gamma_{k-1}^2-\zeta \right) - \frac{d}{\beta} \gamma_{k-1}^2 \\
& = \frac{\gamma_{k-1}^2}{\beta}(\beta-d)+   \frac{\gamma_{k-1}}{m}\left(\lambda d \gamma_{k-1}^2-\zeta\right)  \leq 0
\end{align*}
Here, the last inequality follows from the assumed condition $\gamma_{k-1} \leq \sqrt{\frac{\zeta }{\lambda d}}$. In a similar fashion we have,
\begin{align*}
g\left( \sqrt{\frac{\zeta }{\lambda d}}\right) & = \frac{\zeta}{\lambda \beta} - \frac{\zeta}{m}\sqrt{\frac{\zeta }{\lambda d}}- \frac{d}{\beta} \gamma_{k-1}^2 +  \frac{\lambda d}{m} \sqrt{\frac{\zeta }{\lambda d}} \gamma_{k-1}^2 \\
& \geq \frac{\zeta}{\lambda \beta} - \frac{d}{\beta} \gamma_{k-1}^2 + \frac{\lambda d \zeta}{m^2} \gamma_{k-1}^2  - \frac{\zeta^2}{m^2}  \\
& = \frac{m^2 \left(\zeta-\lambda d \gamma_{k-1}^2\right)  - \lambda \beta \zeta \left(\zeta-\lambda d \gamma_{k-1}^2\right)}{\lambda m^2 \beta} \\
& \geq \frac{\left(m^2 - \lambda \beta \zeta \right)\left(\zeta-\lambda d \gamma_{k-1}^2\right)}{\lambda m^2 \beta} \geq 0
\end{align*}
In this case, we use the identity $\sqrt{\frac{\zeta }{\lambda d}} > \frac{\zeta}{m}$ and $\gamma_{k-1} \leq \sqrt{\frac{\zeta }{\lambda d}}$. This proves the statement,  $\gamma_k \in [\gamma_{k-1},\sqrt{\frac{\zeta }{\lambda d}}]$.
\end{proof}

\begin{remark}
Note that by taking limits as $\lambda \rightarrow 0 ^+$ in Theorem \ref{th:1} we have,
\allowdisplaybreaks{
\begin{align*}
  & \lim_{\lambda \rightarrow 0^+} \frac{4 \lambda \|x_0-x^*\|^2_{(A^TA)^{\dagger}} }{ \zeta \left(\sigma_1^{k+1} - \sigma_2^{k+1}\right)^2}   \\
  = & \lim_{\lambda \rightarrow 0^+} \frac{4 \lambda \|x_0-x^*\|^2_{(A^TA)^{\dagger}} }{ \zeta \left(\left(1+(k+1)\frac{\sqrt{\lambda \beta \zeta}}{2m}+ o \left(\sqrt{\lambda \beta \zeta}\right)\right) - \left(1-(k+1)\frac{\sqrt{\lambda \beta \zeta}}{2m}+o \left(\sqrt{\lambda \beta \zeta}\right)\right)\right)^2} \\
  = & \lim_{\lambda \rightarrow 0^+} \frac{4 \lambda \|x_0-x^*\|^2_{(A^TA)^{\dagger}} }{  \left((k+1)\frac{\sqrt{\lambda \beta }}{m}+ o \left(\sqrt{\lambda \beta \zeta}\right)\right)^2} \\
  = & \frac{4 m^2 \|x_0-x^*\|^2_{(A^TA)^{\dagger}} }{ \beta^2 \left(k+1\right)^2} 
\end{align*}}
Therefore, we can conclude that when $\lambda > 0$, the ASKM algorithm converges with a linear rate. When $\lambda = 0$, we get a sublinear convergence. But for the case of $\lambda \rightarrow 0 ^+$, we get a quadratic convergence, which is consistent with the convergence rate of the original accelerated algorithm of Nesterov \cite{nesterov:2014} and also with the \textit{Accelerated Randomized Kaczmarz} algorithm proposed in \cite{wright:2016}. Furthermore, if we take $\beta =1$, we get exactly the same convergence theorem proven in \cite{wright:2016}.

\begin{proof} (\textbf{Theorem \ref{th:1}})
The proof of theorem \ref{th:1} is general in the context of acceleration. We follow the standard notation and steps shown in \cite{nesterov:2012}, \cite{wright:2016}. Using the definitions given in Lemma \ref{lem:5}, we note that the following relation holds:
\begin{align}
\label{eq:23}
  \frac{1-\alpha_k}{\alpha_k} = \frac{m \gamma_{k-1}^2}{\gamma_k}
\end{align}

Now, let us define $r_{k}^2  = \|v_{k}-x^*\|^2_{(A^TA)^{\dagger}}$. We can write,
\begin{align}
r_{k+1}^2 & = \|v_{k+1}-x^*\|^2_{(A^TA)^{\dagger}} \nonumber \\
&= \|\beta_kv_k+(1-\beta_k)y_k-\gamma_k a_{i^*}(a_{i^*}^Ty_k-b_{i^*})^{+}-x^*\|^2_{(A^TA)^{\dagger}} \nonumber \\
& = \|\beta_kv_k+(1-\beta_k)y_k-x^*\|^2_{(A^TA)^{\dagger}} + \gamma_k^2 \|a_{i^*}(a_{i^*}^Ty_k-b_{i^*})^{+}\|^2_{(A^TA)^{\dagger}} \nonumber \\
& - 2 \gamma_k \inp[\Big]{\beta_kv_k+(1-\beta_k)y_k-x^*}{ (A^TA)^{\dagger} a_{{i^*}}\left(a_{i^*}^Ty_k-b_{i^*}\right)^{+}} \label{eq:24}
\end{align}
Now, we divide the RHS of equation \eqref{eq:24} into three parts and simplify them separately. Since $\|.\|^2_{(A^TA)^{\dagger}}$ is a convex function and $0 \leq \beta_k \leq 1$, $1^{st}$ part of \eqref{eq:24} satisfies the following inequality:
\begin{align}
\label{eq:25}
\|\beta_kv_k+(1-\beta_k) & y_k-x^*\|^2_{(A^TA)^{\dagger}} \nonumber \\ & \leq \beta_k\|v_k-x^*\|^2_{(A^TA)^{\dagger}} + (1-\beta_k)\|y_k-x^*\|^2_{(A^TA)^{\dagger}} \nonumber \\
& \leq \beta_k\|v_k-x^*\|^2_{(A^TA)^{\dagger}} + \frac{(1-\beta_k)}{\lambda}\|y_k-x^*\|^2
\end{align}
Let us denote $i(k)$ as the index which represents the random selection at iteration $k$. And let $I(k)$ denote all random indices occurred before or at iteration $k$, i.e.,
\begin{align}
\label{eq:26}
 I(k) = \{i(k), i(k-1),...,i(0)\} \end{align}

The sequences $x_{k+1}, y_{k+1}, v_{k+1}$ are dependent on $I(k)$. In the next part of the proof, we use $\E_{i(k)|I(k-1)}[.]$ to represent the expectation of a random variable conditioned on $I(k-1)$ with respect to the index $i(k)$. Note that,
\begin{align}
\label{eq:27}
\E_{I(k)}[.] = \E_{I(k-1)}\left[\E_{i(k)|I(k-1)}[.]\right]    
\end{align}

Also note that, from now on we use $\E$ instead ($\E_{\beta i^*}$) to denote the expectation. Now, based on the Lemma \ref{lem:2} and Lemma \ref{lem:3}, we can write the $2^{nd}$ part of \eqref{eq:24} as follows:
\allowdisplaybreaks{
\begin{align}
\E_{i(k)| I(k-1)} & \left[\big\|a_{i^*}\left(a_i^Ty_k-b_i\right)^{+}_{i^*}\big\|^2_{(A^TA)^{\dagger}}\right]  \nonumber \\ 
& \leq \frac{\beta}{m} \ \| (Ay_k-b)^{+}\|^2 \nonumber \\
& \leq \ \beta\|y_k-x^*\|^2-\beta \E \left[\big\|\mathcal{P}_{a_{i^*}, b_{i^*}}(y_k)-x^*\big\|^2\right] \nonumber \\
& \leq \ \beta\|y_k-x^*\|^2-\beta \E_{i(k)| I(k-1)} \left[\big\|x_{k+1}-x^*\big\|^2\right] \label{eq:28}
\end{align}}
Now, by using the definitions of the sequences $\{\alpha_k\}, \{\beta_k\}$ and $\{\gamma_k\}$, we can simply show that the following identity holds: 
\begin{align}
\label{eq:29}
\frac{\zeta \beta}{m} \frac{1-\alpha_k}{\alpha_k}\beta_k  = \frac{\zeta \beta }{m}  z_k = d \beta_k \frac{\gamma_{k-1}^2}{\gamma_k}
\end{align}
We use the identity of \eqref{eq:29} in the next part of our proof. After taking expectation in the third term of equation \eqref{eq:24}, we get,
\begin{align}
& \E_{i(k)|I(k-1)} \left[2 \gamma_k \inp[\Big]{x^*-y_k+ \frac{1-\alpha_k}{\alpha_k}\beta_k(x_k-y_k)}{ (A^TA)^{\dagger} a_{{i^*}}\left(a_{i^*}^Ty_k-b_{i^*}\right)^{+}}\right] \nonumber \\
& = \ 2 \gamma_k  \inp {x^*-y_k+ \frac{1-\alpha_k}{\alpha_k}\beta_k(x_k-y_k)}{ (A^TA)^{\dagger}\E_{i(k)|I(k-1)}\left[ a_{{i^*}}(a_{i^*}^Ty_k-b_{i^*})^{+}\right]} \nonumber \\
 = \ & \frac{2 \gamma_k}{\binom{m}{\beta}} \inp {x^*-y_k+ \frac{1-\alpha_k}{\alpha_k}\beta_k(x_k-y_k)}{ (A^TA)^{\dagger} \sum\limits_{j = 0}^{m-\beta} \binom{\beta-1+j}{\beta-1} a_{i_j} (a_{i_j}^Ty_k-b_{i_j})^{+}} \nonumber \\
& \ \ \leq \frac{2 \gamma_k \beta}{m} \inp[\Big]{x^*-y_k+ \frac{1-\alpha_k}{\alpha_k}\beta_k(x_k-y_k)}{ (A^TA)^{\dagger} A^T (Ay_k-b)^+} \label{eq:30}
\end{align}

\noindent  Using the definition of the function $f(.)$ defined in \eqref{eq:18} and denoting $z_k = \beta_k \frac{1-\alpha_k}{\alpha_k}$, we get,
\begin{align}
& \inp[\Big]{x^*-y_k+ z_k(x_k-y_k)}{ (A^TA)^{\dagger} A^T (Ay_k-b)^+} \nonumber \\
& = \inp[\Big]{x^*+ z_k(x_k-y_k)-y_k}{ \nabla f(y_k)} \nonumber \\
& \leq f\left(x^*+ z_k(x_k-y_k)\right) -f(y_k) - \frac{\zeta}{2} \|x^*+ z_k(x_k-y_k)-y_k\|^2 \nonumber \\
& \leq \frac{z_k^2}{2} \left[(Ax_k-Ay_k)^+\right]^T A(A^TA)^{\dagger}A^T (Ax_k-Ay_k)^+ + \zeta z_k \langle y_k-x^*, x_k-y_k \rangle \nonumber \\
& - \frac{\zeta}{2} z_k^2 \|x_k-y_k\|^2 - \frac{\zeta}{2}  \|y_k-x^*\|^2  - \frac{1}{2}\left[(Ay_k-b)^+\right]^T A (A^TA)^{\dagger} A^T (Ay_k-b)^+  \nonumber \\
& \leq \frac{1}{2}   \|x_k-y_k\|^2 \left[\lambda_{\max}^2 z_k^2-\zeta z_k^2 - \zeta z_k \right] + \frac{\zeta z_k}{2}  \|x_k-x^*\|^2 \nonumber \\
& - \frac{1}{2} \|y_k-x^*\|^2 \left[1 +\zeta + \zeta z_k \right] \nonumber \\
& \leq  -\frac{1}{2} \zeta z_k \|x_k-y_k\|^2  + \frac{\zeta}{2} z_k \|x_k-x^*\|^2 - \frac{1}{2} \|y_k-x^*\|^2 \left[1 +\zeta + \zeta z_k \right] \nonumber \\
&  + \frac{1}{2}  \frac{z_k^2 \lambda_{\max}^2}{\lambda_{\min}^2} \|x_k-y_k\|^2 \left[\lambda_{\min}^2 -1 \right] \nonumber \\
& \leq  -\frac{\zeta}{2}  z_k \|x_k-y_k\|^2  + \frac{\zeta}{2} z_k \|x_k-x^*\|^2 - \frac{1}{2} \|y_k-x^*\|^2 \left[1 +\zeta + \zeta z_k \right] \label{eq:31}
\end{align}

\noindent Now, substituting equation \eqref{eq:31} in \eqref{eq:30}  with the known identity, $\frac{\zeta \beta \gamma_k}{m}  z_k = d \beta_k \gamma_{k-1}^2$, we have,
\begin{align}
& \E_{i(k)|I(k-1)} \left[2 \gamma_k \inp[\Big]{x^*-y_k+ \frac{1-\alpha_k}{\alpha_k}\beta_k(x_k-y_k)}{ (A^TA)^{\dagger} a_{{i^*}}\left(a_{i^*}^Ty_k-b_{i^*}\right)^{+}}\right] \nonumber \\
& \leq \frac{2 \gamma_k \beta}{m} \inp[\Big]{x^*-y_k+ \frac{1-\alpha_k}{\alpha_k}\beta_k(x_k-y_k)}{ (A^TA)^{\dagger} A^T (Ay_k-b)^+} \nonumber \\
& \leq  -\frac{\zeta \beta \gamma_k}{m}  z_k \|x_k-y_k\|^2  + \frac{\zeta \beta \gamma_k}{m} z_k \|x_k-x^*\|^2 - \frac{\beta \gamma_k}{m} \|y_k-x^*\|^2 \left[1 +\zeta + \zeta z_k \right] \nonumber \\
& \leq d \beta_k \gamma_{k-1}^2 \|x_k-x^*\|^2 - \frac{\beta \gamma_k}{m} \|y_k-x^*\|^2 \left[1 +\zeta + \zeta z_k \right] \label{eq:32}
\end{align}

\noindent Now by substituting all three parts of \eqref{eq:25}, \eqref{eq:28} and \eqref{eq:32} in equation \eqref{eq:24}, we get,
\begingroup
\allowdisplaybreaks
\begin{align}
\E_{i(k)| I(k-1)} & \left(r_{k+1}^2 \right) \nonumber \\
& = \beta_k\|v_k-x^*\|^2_{(A^TA)^{\dagger}} + \frac{\beta}{m} \gamma_k \|y_k-x^*\|^2 + \beta \gamma_k^2 \|y_k-x^*\|^2 \nonumber \\
& -\beta \gamma_k^2 \E_{i(k)| I(k-1)} \left[\big\|x_{k+1}-x^*\big\|^2\right] + d \beta_k \gamma_{k-1}^2 \|x_k-x^*\|^2 \nonumber \\
& - \frac{\beta \gamma_k}{m} \|y_k-x^*\|^2 \left[1 +\zeta + \zeta z_k \right] \nonumber \\
& \leq \beta_k\|v_k-x^*\|^2_{(A^TA)^{\dagger}} -\beta \gamma_k^2 \E_{i(k)| I(k-1)} \left[\big\|x_{k+1}-x^*\big\|^2\right] \nonumber \\
& + \beta \left(\gamma_k^2-\frac{\zeta}{m} \gamma_k - \frac{d}{\beta} \beta_k \gamma_{k-1}^2\right) \|y_k-x^*\|^2 + d \beta_k \gamma_{k-1}^2 \|x_k-x^*\|^2 \nonumber \\
& = \beta_k\|v_k-x^*\|^2_{(A^TA)^{\dagger}} +  d \beta_k \gamma_{k-1}^2 \|x_k-x^*\|^2 \nonumber \\
& - \beta \gamma_k^2 \E_{i(k)| I(k-1)} \left[\big\|x_{k+1}-x^*\big\|^2\right] \label{eq:33}
\end{align}
\endgroup

\noindent From now on, we will assume $d = \beta$, which will simplify our algorithm. Let us define two sequences $\{A_k\}$ and $\{B_k\}$ as follows:
\begin{align}
\label{eq:34}
A_k \geq 0, B_k \geq 0, B_0 \neq 0, B_{k+1}^2 = \frac{B_k^2}{\beta_k}, A_{k+1}^2 = \beta \gamma_k^2 B_{k+1}^2 
\end{align}
Without loss of generality, we assume $A_0 = 0$ to be consistent with the definition $\gamma_{-1} = 0$. Also note that since $\beta_k \in (0,1]$, we have $B_{k+1} \geq B_k$. Now using the definition of the sequence $\{\gamma_k\}, \{\alpha_k\}$, we have,
\begin{align}
\label{eq:35}
A_{k+1}^2 = \beta \gamma_k^2\frac{B_k^2}{\beta_k} = \frac{ \gamma_k^2}{ \beta_k \gamma_{k-1}^2} A_k^2 = \frac{A_k^2 \gamma_k^2}{\gamma_k^2-\frac{\zeta}{m}\gamma_k}
\end{align}
Equation \eqref{eq:35} also implies that the sequence $\{A_k\}$ is an increasing sequence. Now, it is straightforward to check that the following identities hold.
\begin{align}
\label{eq:36}
\beta B_{k+1}^2 \gamma_k^2 = A_{k+1}^2, \  B_{k+1}^2 \beta_k = B_{k}^2 \ \ \text{and} \ B_{k+1}^2 \beta \beta_k \gamma_{k-1}^2 = A_k^2
\end{align}
Now, multiplying both sides of \eqref{eq:33} by $B_{k+1}^2$ and using the above identities we have,
\begin{align}
B_{k+1}^2 \E_{i(k)| I(k-1)} & \left(r_{k+1}^2 \right) +  A_{k+1}^2 \E_{i(k)| I(k-1)} \left[\big\|x_{k+1}-x^*\big\|^2\right] \nonumber \\
& \leq \  B_{k}^2 r_k^2 + A_{k}^2  \|x_k-x^*\|^2 \label{eq:37}
\end{align}
Furthermore, we have,
\begin{align}
\E_{I(k)} & \left[B_{k+1}^2 r_{k+1}^2 + A_{k+1}^2 \big\|x_{k+1}-x^*\big\|^2\right] \nonumber \\
 = & \E_{I(k-1)} \left[B_{k+1}^2 \E_{i(k)| I(k-1)} \left(r_{k+1}^2 \right) +  A_{k+1}^2 \E_{i(k)| I(k-1)} \left[\big\|x_{k+1}-x^*\big\|^2\right]\right] \nonumber \\
& \leq  \ \E_{I(k-1)} \left[B_{k}^2 r_k^2 + A_{k}^2  \|x_k-x^*\|^2\right] \nonumber \\
& \ \ \vdots \nonumber \\
& \leq  \ \E_{I(0)} \left[B_{1}^2 r_1^2 + A_{1}^2 \|x_1-x^*\|^2\right] \nonumber \\
& \leq \ B_0^2 r_0^2 + A_{0}^2  \|x_0-x^*\|^2 \nonumber \\
& = \ B_0^2 r_0^2 \label{eq:38}
\end{align}
Therefore, using \eqref{eq:38} we can conclude the following bound,
\begin{align}
\label{eq:39}
\E\left[r_{k+1}^2\right] \ \leq \ \frac{B_0^2}{B_{k+1}^2} r_0^2 \ \ \text{and} \ \ \E\left[\|x_{k+1}-x^*\|^2\right] \ \leq \ \frac{B_0^2}{A_{k+1}^2} r_0^2
\end{align}

Now, we need to estimate the growth of the defined sequences  $\{A_k\}$ and $\{B_k\}$. Here, we follow the proof for the \textit{Accelerated Coordinate Descent} method of Nesterov \cite{nesterov:2012} and accelerated randomized Kaczmarz algorithm by Wright \textit{et. al} \cite{wright:2016} as they are more general in the context of acceleration. We have,
\begin{align}
\label{eq:40}
B_k^2 = \beta_k B_{k+1}^2 = \left(1- \frac{\lambda \beta}{m}\gamma_k\right) B_{k+1}^2 = \left(1- \frac{\lambda \beta}{m}\frac{1}{\sqrt{\beta}} \frac{A_{k+1}}{B_{k+1}}\right) B_{k+1}^2
\end{align}
Simplifying \eqref{eq:40} we get,
\begin{align*}
\frac{\lambda \sqrt{\beta}}{m} A_{k+1}B_{k+1} & = B^2_{k+1}-B_k^2 = (B_{k+1}-B_k)(B_{k+1}+B_k) \\
& \leq 2 B_{k+1}(B_{k+1}-B_k)
\end{align*}
Here, we used the identity $B_{k+1} \geq B_k$, which simplifies to:
\begin{align}
\label{eq:41}
  B_{k+1} \geq B_k + \frac{\lambda \sqrt{\beta}}{2m} A_{k+1} \geq B_k + \frac{\lambda \sqrt{\beta}}{2m} A_{k}   
\end{align}
Similarly, note that,
\begin{align*}
 \frac{A_{k+1}^2}{B_{k+1}^2} - \frac{\zeta \sqrt{\beta}}{m} \frac{A_{k+1}}{B_{k+1}} & = \beta \gamma_k^2 - \sqrt{\beta} \gamma_k \frac{\zeta \sqrt{\beta}}{m} \\ 
 & = \beta \left(1- \frac{\lambda \beta}{m} \gamma_k\right) \gamma_{k-1}^2 \\
 & = \beta \beta_k \gamma_{k-1}^2 = \beta_k \frac{A_k^2}{B_k^2} = \frac{A_k^2}{B_{k+1}^2}
\end{align*}
Above equation simplifies to the following:
\begin{align*}
\frac{\zeta \sqrt{\beta}}{m} A_{k+1}B_{k+1} & = A^2_{k+1}-A_k^2 = (A_{k+1}-A_k)(A_{k+1}+A_k) \\
& \leq 2 A_{k+1}(A_{k+1}-A_k)
\end{align*}
In this case, we used the identity $A_{k+1} \geq A_k$, which leads to the following identity:
\begin{align}
\label{eq:42}
  A_{k+1} \geq A_k + \frac{\zeta \sqrt{\beta}}{2m} B_{k+1} \geq A_k + \frac{\zeta \sqrt{\beta}}{2m} B_{k} 
\end{align}
By combining the two expressions of \eqref{eq:41} and \eqref{eq:42} in a LS we get,
\begin{align}
\label{eq:43}
 \begin{bmatrix}
A_{k+1} \\
B_{k+1} \\
\end{bmatrix} \geq  \begin{bmatrix}
1 & \frac{\zeta \sqrt{\beta}}{2m} \\
\frac{\lambda \sqrt{\beta}}{2m} & 1 \\
\end{bmatrix}  \begin{bmatrix}
A_{k} \\
B_{k} \\
\end{bmatrix} \geq \hdots \geq \  \begin{bmatrix}
1 & \frac{\zeta \sqrt{\beta}}{2m} \\
\frac{\lambda \sqrt{\beta}}{2m} & 1 \\
\end{bmatrix}^{k+1}  \begin{bmatrix}
A_{0} \\
B_{0} \\
\end{bmatrix}
\end{align}
The Jordan decomposition of the matrix in the above expression is given by,
\begin{align}
\label{eq:44}
\begin{bmatrix}
1 & \frac{\zeta \sqrt{\beta}}{2m} \\
\frac{\lambda \sqrt{\beta}}{2m} & 1 \\
\end{bmatrix} =  \begin{bmatrix}
-\sqrt{\frac{\zeta}{\lambda}} & \sqrt{\frac{\zeta}{\lambda}} \\
1 & 1 \\
\end{bmatrix} \begin{bmatrix}
\sigma_2 & 0 \\
0 & \sigma_1 \\
\end{bmatrix}   \begin{bmatrix}
-\frac{1}{2}\sqrt{\frac{\lambda}{\zeta}} & \frac{1}{2} \\
\frac{1}{2}\sqrt{\frac{\lambda}{\zeta}} & \frac{1}{2} \\
\end{bmatrix}
\end{align}
Here, $\sigma_1 = 1+\frac{\sqrt{\lambda \beta \zeta}}{2m}$ and $\sigma_2 = 1-\frac{\sqrt{\lambda \beta \zeta}}{2m}$. Using $A_0 = 0$ and the decomposition of \eqref{eq:44}, from equation \eqref{eq:43} we have,
\begin{align*}
 \begin{bmatrix}
A_{k+1} \\
B_{k+1} \\
\end{bmatrix} & \geq \  \begin{bmatrix}
1 & \frac{\zeta \sqrt{\beta}}{2m} \\
\frac{\lambda \sqrt{\beta}}{2m} & 1 \\
\end{bmatrix}^{k+1}  \begin{bmatrix}
A_{0} \\
B_{0} \\
\end{bmatrix} \\
& = \begin{bmatrix}
-\sqrt{\frac{\zeta}{\lambda}} & \sqrt{\frac{\zeta}{\lambda}} \\
1 & 1 \\
\end{bmatrix} \begin{bmatrix}
\sigma_2^{k+1} & 0 \\
0 & \sigma_1^{k+1} \\
\end{bmatrix}   \begin{bmatrix}
-\frac{1}{2}\sqrt{\frac{\lambda}{\zeta}} & \frac{1}{2} \\
\frac{1}{2}\sqrt{\frac{\lambda}{\zeta}} & \frac{1}{2} \\
\end{bmatrix} \begin{bmatrix}
0 \\
B_{0} \\
\end{bmatrix} \\
& = \begin{bmatrix}
 \frac{1}{2}\sqrt{\frac{\zeta}{\lambda}} \left(\sigma_1^{k+1}- \sigma_2^{k+1}\right) B_0 \\
\frac{1}{2} \left(\sigma_1^{k+1}+ \sigma_2^{k+1}\right) B_0  \\
\end{bmatrix}
\end{align*}
The above gives us the following growth bound for the sequences $\{A_k\}$ and $\{B_k\}$ as follows:
\begin{align}
& A_{k+1} \geq   \frac{1}{2}\sqrt{\frac{\zeta}{\lambda}} \left(\sigma_1^{k+1}- \sigma_2^{k+1}\right) B_0 \label{eq:45} \\
& B_{k+1} \geq \frac{1}{2} \left(\sigma_1^{k+1}+ \sigma_2^{k+1}\right) B_0 \label{eq:46}
\end{align}
Substituting these above bounds of \eqref{eq:45} and \eqref{eq:46} in equation \eqref{eq:39}, we get the following bounds:
\begin{align}
\E \left(r_{k+1}^2\right) & = \E \left[\|v_{k}-x^*\|^2_{(A^TA)^{\dagger}}\right] \leq \frac{B_0^2}{B_{k+1}^2} \leq \frac{4 \|x_0-x^*\|^2_{(A^TA)^{\dagger}} }{\left(\sigma_1^{k+1}+ \sigma_2^{k+1}\right)^2} \label{eq:47} \\
\E & \left[\|x_{k+1}-x^*\|^2\right]  \leq \frac{B_0^2}{A_{k+1}^2} r_0^2 \leq \  \frac{4 \lambda \|x_0-x^*\|^2_{(A^TA)^{\dagger}} }{ \zeta \left(\sigma_1^{k+1} - \sigma_2^{k+1}\right)^2}  \label{eq:48}
\end{align}
The above equations complete the proof of Theorem \ref{th:1}.
\end{proof}

\end{remark}

\section{Numerical Experiments}
\label{sec:num}
We implemented the ASKM algorithm in MATLAB and performed the numerical experiments in a Dell Precision 7510 workstation with 32GB RAM, Intel Core i7-6820HQ CPU, processor running at 2.70 GHz. We divided the numerical experiments into three categories: experiments on randomly generated problems, experiments on real-world non-random problems and comparison among different methods. In these experiments, we compared ASKM with SKM and other state-of-the-art methods (i.e., IPM and ASM). As mentioned earlier, our main focus is on the over-determined systems regime (i.e., $m \gg n$), where iterative methods are applied in general. For all of the experiments, we ran the algorithms 10 times and report the averaged performance.

\subsection{Comparison of SKM and ASKM on random data}
\label{subsec:1}
We considered systems $Ax \leq b$ where the entries of $A$ and $b$ are chosen randomly from the corresponding distribution. To make sure that $b \in \mathcal{R}(\mathbf{A})$, we generated two vectors $x_1, x_2 \in \R^n$ at random from the corresponding distributions, then multiplied them by $A$ and set $b$ as a convex combination of those two vectors. We considered two types of random data sets: highly correlated systems and Gaussian systems. In the highly correlated systems, entries of $A$ are chosen uniformly at random between $[0.9,1.0]$ and $b$ is chosen accordingly such that the system $Ax \leq b$ has a feasible solution. The entries of $A$ in the Gaussian systems are chosen from standard normal distribution and $b$ is chosen accordingly as before.

In Figure \ref{fig:1}, we provide a comparison between SKM and ASKM for three randomly generated correlated systems. We compare the average computational time necessary for SKM and ASKM with several choices of sample size $\beta$ to reach positive residual error $10^{-05}$ (i.e., $\|\left(Ax_k-b\right)^+\|_2 \leq 10^{-05}$). We compare the two algorithms for the choice of \footnote{$\delta$ here is same as $\lambda$ in De Loera \textit{et. al} \cite{haddock:2017}} $\delta = 1$. For the three test cases, we see that for any $1 \leq \beta \leq m$, ASKM significantly outperform SKM in terms of average computation time.

\begin{figure}[H]
\begin{subfigure}{0.315\textwidth}
\includegraphics[width=\linewidth]{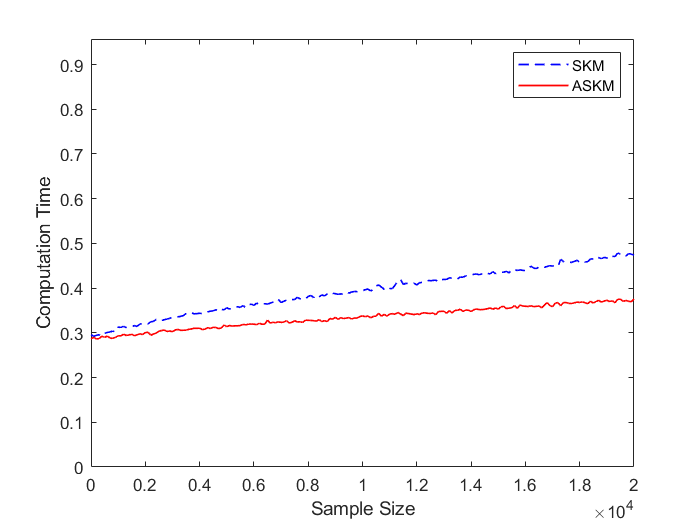}
\caption{$20,000 \times 1,000$}
\end{subfigure}
\hspace*{\fill}
\begin{subfigure}{0.31\textwidth}
\includegraphics[width=\linewidth]{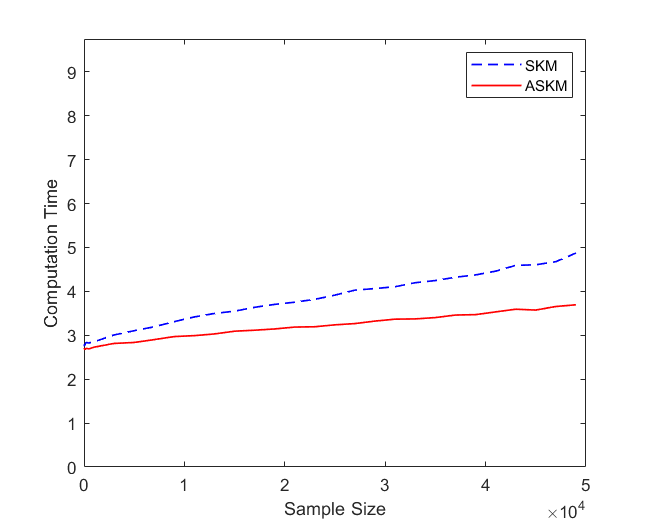}
\caption{$50,000 \times 4,000$}
\end{subfigure}
\hspace*{\fill}
\begin{subfigure}{0.32\textwidth}
\includegraphics[width=\linewidth]{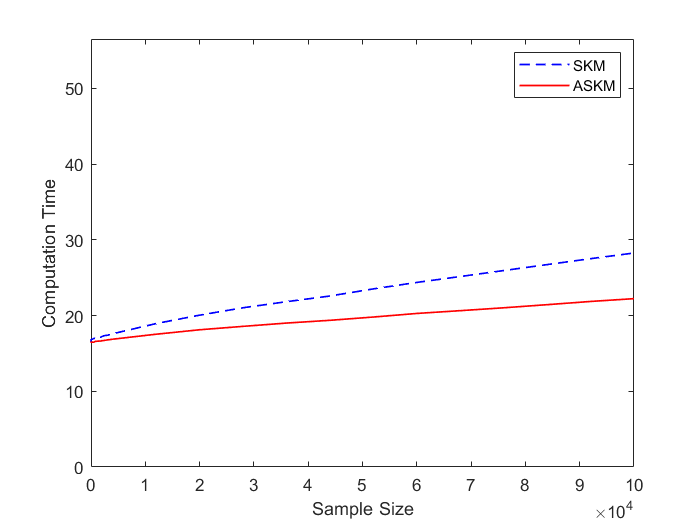}
\caption{$100,000 \times 10,000$}
\end{subfigure}
\caption{Average computation time of SKM and ASKM on highly correlated random systems to reach residual error $\|\left(Ax_k-b\right)^+\|_2 \leq 10^{-05}$.} \label{fig:1}
\end{figure}

In Figure \ref{fig:2}, we show the same comparison experiments for randomly generated Gaussian systems. Similar to the correlated systems, ASKM algorithm solves the Gaussian systems much faster than the SKM algorithm. Notice that in Figure \ref{fig:2}(b) the computational time of SKM algorithm stays at 1000 seconds for sample size $\beta \geq 2000$. This happens due to an additional terminating condition of maximum run time set at 1000 seconds. While SKM algorithm fails to converge within the limiting time for larger sample sizes ($\beta \geq 2000)$, ASKM algorithm finds a feasible solution for any sample size. Moreover, if we analyze the trend of ASKM's average computational time in both figures (Figure \ref{fig:1} and \ref{fig:2}), we see that ASKM accelerates the SKM algorithm and the nature of acceleration is quadratic which validates our claim of Theorem \ref{th:1}.

\begin{figure}[H]
\begin{subfigure}{0.49\textwidth}
\includegraphics[width=\linewidth]{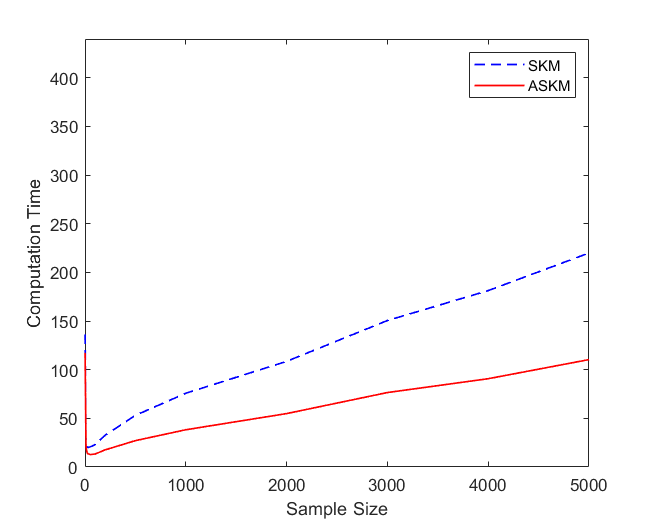}
\caption{$5,000 \times 1,000$}
\end{subfigure}
\hspace*{\fill}
\begin{subfigure}{0.49\textwidth}
\includegraphics[width=\linewidth]{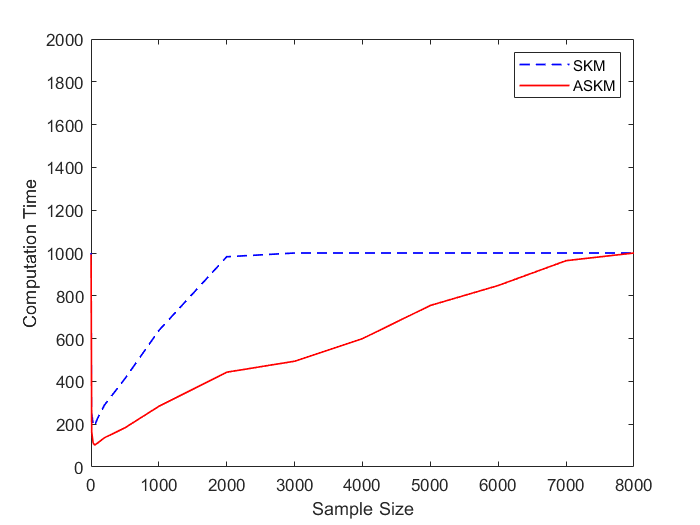}
\caption{$8,000 \times 2,000$} 
\end{subfigure}
\caption{Average computation time of SKM and ASKM on Gaussian random systems to either reach residual error $\|\left(Ax_k-b\right)^+\|_2 \leq 10^{-05}$ or a maximum run time limit of 1000 seconds.} \label{fig:2}
\end{figure}
\vspace{-0.5 cm}

In Figure \ref{fig:3} and \ref{fig:4}, we compare the positive residual error for SKM and ASKM for different sample sizes ($\beta = 1, 100, 1000$). We plot iteration versus residual error and time versus residual error for random Gaussian systems. Based on the findings of Figure \ref{fig:3} and \ref{fig:4}, we can conclude that irrespective of sample size selection, $\|(Ax_k-b)^+\|_2$ converges to zero much more faster for ASKM than for SKM. The convergence of $\|(Ax_k-b)^+\|_2$ for both ASKM and SKM are much slower for the choice of $\beta = 1$ as expected. 

\begin{figure}[H]
\begin{subfigure}{0.50\textwidth}
\includegraphics[width=\linewidth]{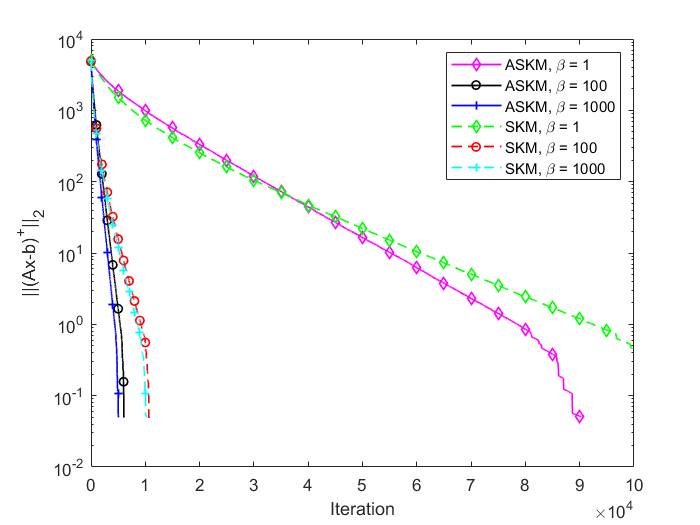}
\caption{Iterations versus residual error}
\end{subfigure}
\hspace*{\fill}
\begin{subfigure}{0.50\textwidth}
\includegraphics[width=\linewidth]{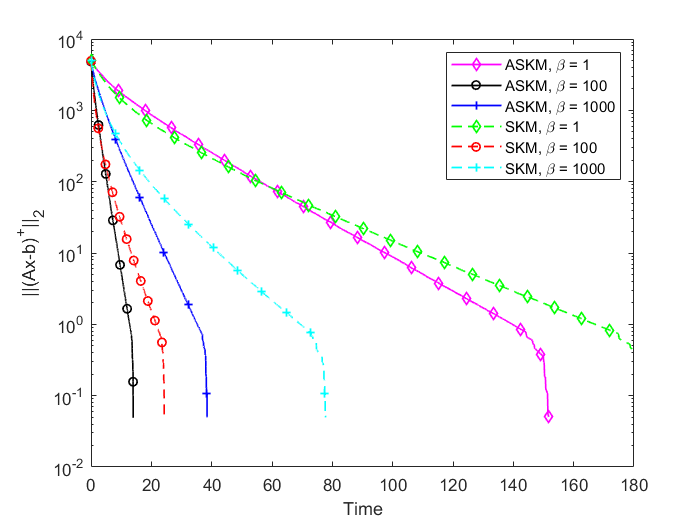}
\caption{Time versus residual error}
\end{subfigure}
\caption{Comparison of SKM and ASKM on a $5000 \times 1000$ random Gaussian system for sample size $\beta = 1, 100, 1000$.} \label{fig:3}
\end{figure}

\begin{figure}[H]
\begin{subfigure}{0.50\textwidth}
\includegraphics[width=\linewidth]{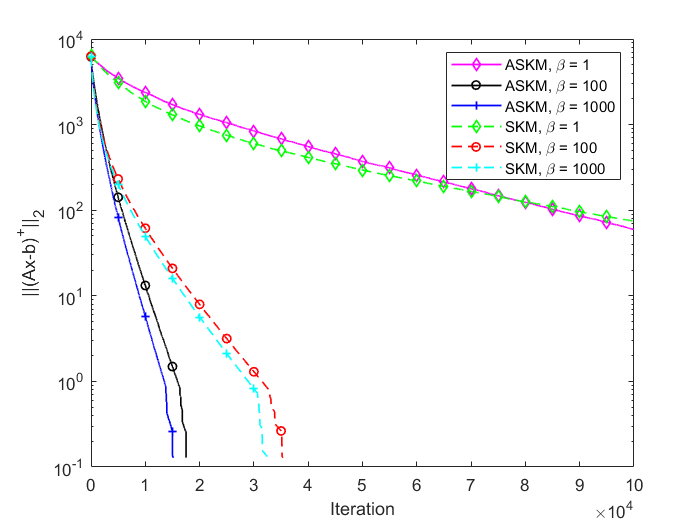}
\caption{Iterations versus residual error}
\end{subfigure}
\hspace*{\fill}
\begin{subfigure}{0.50\textwidth}
\includegraphics[width=\linewidth]{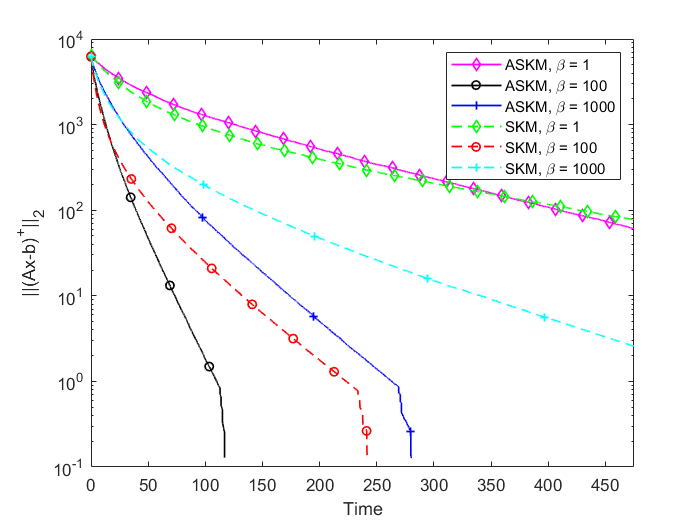}
\caption{Time versus residual error}
\end{subfigure}
\caption{Comparison of SKM and ASKM on a $8000 \times 2000$ random Gaussian system for sample size $\beta = 1, 100, 1000$.} \label{fig:4}
\end{figure}

For $\beta = 100$ and $\beta = 1000$, the convergence rate of ASKM takes over SKM after a small amount of time. In addition, the convergence rate remains similar for both the test case problems ($5000 \times 1000$ and $8000 \times 2000$). To investigate the solution quality of both SKM and ASKM, we measure the number of satisfied constraints after each iteration and the corresponding computational time for both algorithms. We summarize our findings in Figure \ref{fig:5} and \ref{fig:6} for the above test sets. For simplification, we used the Fraction of satisfied constraints (FSC) as a measure of quality of the solution generated by both SKM and ASKM algorithms. After analyzing Figures \ref{fig:5} and \ref{fig:6}, we can conclude that the choice of $\beta = 1$ is the worst choice as both SKM and ASKM takes much more time to satisfy all the constraints. However, for the choice of $\beta = 100$ and $\beta = 1000$, ASKM takes much less time compared to SKM to find a solution within the error margin. For example, in Figure \ref{fig:5}, the choice of $\beta = 1000$ ASKM takes approximately 37 seconds to satisfy all the 5000 constraints whereas SKM takes up to 75 seconds.
\begin{figure}[H]
\begin{subfigure}{0.49\textwidth}
\includegraphics[width=\linewidth]{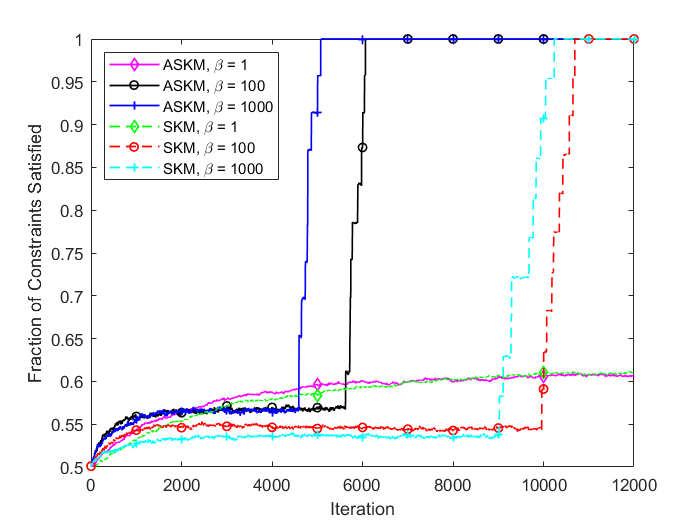}
\caption{Iterations versus FSC}
\end{subfigure}
\hspace*{\fill}
\begin{subfigure}{0.50\textwidth}
\includegraphics[width=\linewidth]{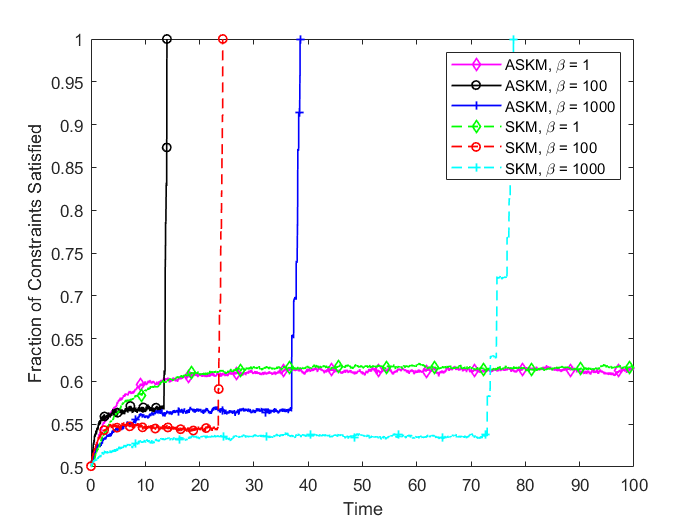}
\caption{Time versus FSC}
\end{subfigure}
\caption{Comparison of SKM and ASKM in terms of fraction of satisfied constraints on a $5000 \times 1000$ random Gaussian system.} \label{fig:5}
\end{figure} 
\begin{figure}[H]
\begin{subfigure}{0.49\textwidth}
\includegraphics[width=\linewidth]{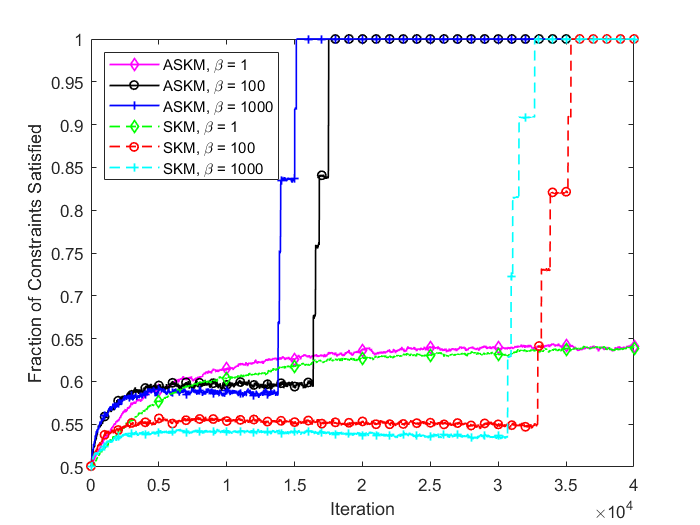}
\caption{Iterations versus FSC}
\end{subfigure}
\hspace*{\fill}
\begin{subfigure}{0.50\textwidth}
\includegraphics[width=\linewidth]{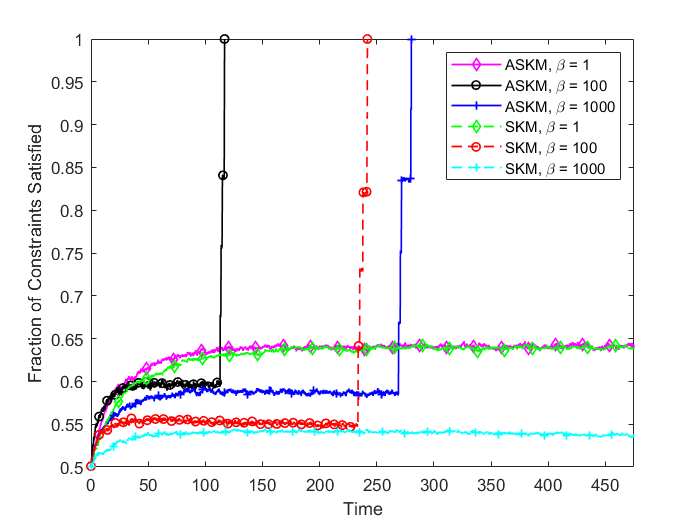}
\caption{Time versus FSC}
\end{subfigure}
\caption{Comparison of SKM and ASKM in terms of fraction of satisfied constraints on a $8000 \times 2000$ random Gaussian system.} \label{fig:6}
\end{figure}

\subsection{Comparison of SKM and ASKM for real-world non-random data}
\label{subsec:2}
In this subsection, we consider two real-world non-random problems. We consider Support Vector Machine (SVM) instances with linear classification and feasibility problems arising in benchmark libraries. We considered the standard test cases given in \cite{calafiore:2014,lichman:2013,haddock:2017}. 

We compare SKM and ASKM methods to solve the linear classification problem with SVM for 1) Wisconsin (diagnostic) breast cancer data set and 2) Credit card data set. The breast cancer data set includes data points whose features are computed from digitized images. Each data point is classified either as malignant or as benign. Our goal is to find a solution of the homogeneous system of inequalities, $Ax \leq 0$ which represents the separating hyperplane between malignant and benign data points. The system of inequalities has 569 constraints (data points) and 30 variables (features). Since the data set is not separable, we set SKM and ASKM to find the minimized residual norm $\|Ax_k\|_2$. For our setup, We consider the threshold $\|Ax_k\|_2 \leq 0.0005$ and $10^{-6}$.

The credit card data set described in \cite{yeh:2009,haddock:2017} consists of features describing the payment profile of user and binary variable for on-time or default payment in a certain billing cycle. Similar to the breast cancer data set, this problem can be solved by finding a solution to the corresponding homogeneous system of inequalities, $Ax \leq 0$ which represents the separating hyperplane between given on-time and default data points. The resulting system of inequalities has 30000 constraints (30000 user profiles) and 23 variables (22 profile features). Since the data set is not separable, we set SKM and ASKM to find the minimized residual norm $\|Ax_k\|_2$. For our setup, we considered the threshold as $\frac{\|Ax_k\|_2}{\|Ax_0\|_2} \leq 0.1$ and $0.001$.

\begin{figure}[t]
\begin{subfigure}{0.50\textwidth}
\includegraphics[width=\linewidth]{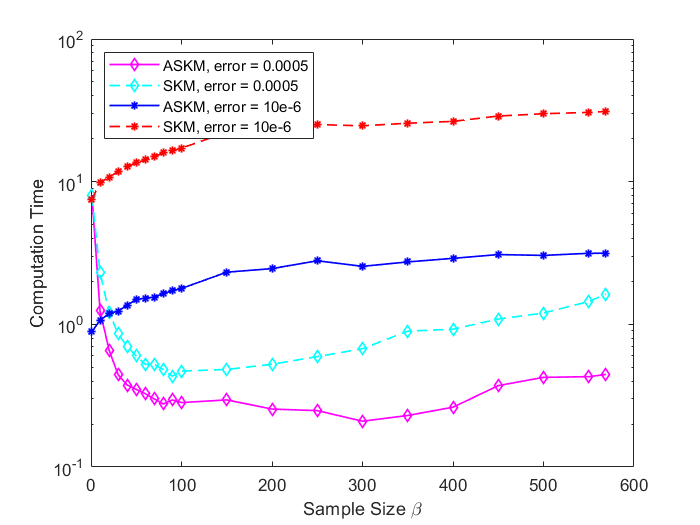}
\caption{Breast cancer data}
\end{subfigure}
\hspace*{\fill}
\begin{subfigure}{0.50\textwidth}
\includegraphics[width=\linewidth]{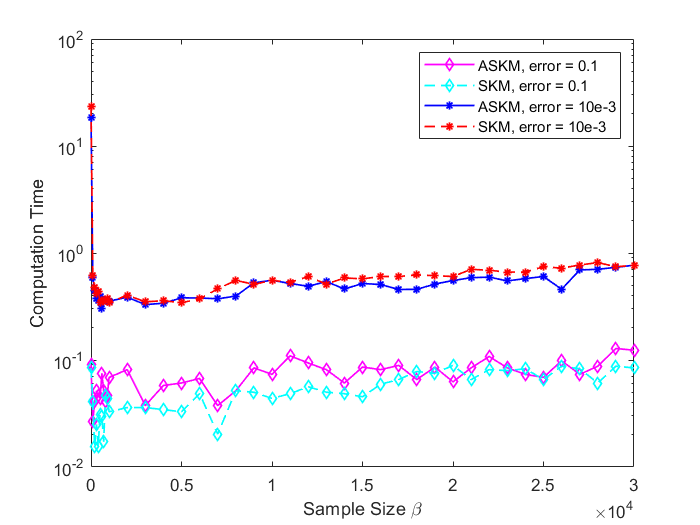}
\caption{Credit card data}
\end{subfigure}
\caption{Comparison of SKM and ASKM for real life data sets} \label{fig:7}
\end{figure} 
Based on the comparison graphs shown in Figure \ref{fig:7}, we can conclude ASKM performs much better than SKM for the breast cancer data set (Figure \ref{fig:7}(a)). For the credit card data set ASKM performs marginally better than SKM for smaller error. Also note that, the computation time curve for credit card data is not as smooth as previous curves, which we can attribute to the irregularity of the coefficients. Such irregularity in the coefficients creates a dependence bias between residual error and actual constraints.

\subsection{Comparison among SKM, ASKM and existing methods for Netlib LP}
\label{subsec:3}
In this subsection, we investigate the comparative performance of the proposed ASKM algorithm with SKM and benchmark algorithms such as IPM and ASM on several Netlib LPs. For the implementation of SKM and ASKM to the Netlib LPs, we follow the framework given by De Loera \textit{et. al} \cite{haddock:2017}. Each of these problems was formulated as a standard LP problem ($\min c^Tx$ subject to $Ax = b, \ l \leq x \leq u$ with optimum value $p^*$). Loera \textit{et. al} \cite{haddock:2017} transformed them into an equivalent LF problem $\bar{A}x \leq \bar{b}$, where $\bar{A} = [A \ -A \ I \ -I \ c^T]^T$ and $\bar{b} = [b \ -b \ u \ -l \ p^*]^T$. We used this setup for all the experiments on Netlib LPs.

In Table \ref{tab:1}, we provide the performance behaviour (computation time in seconds) of ASKM, SKM, IPM and ASM on the Netlib LPs. For fair comparison, we coded ASKM, SKM in MATLAB and compared with the MATLAB Optimization Toolbox function \texttt{fmincon}. Note that \texttt{fmincon} allows us to select both IPM and ASM methods.

\begin{table}[H]
\small{
\begin{center}
\caption{CPU time comparisons among MATLAB methods solving LP, and SKM and ASKM solving LF. $^*$ indicates that the solver did not solve the problem to the desired accuracy due to reaching an predetermined upper limit on function evaluations of 100000.}
\label{tab:1}
\begin{tabular}{@{}|c|c|c|c|c|c|c|c|@{}}
\hline
Instance      & Dimensions & ASKM   & SKM    & \begin{tabular}[c]{@{}c@{}}Interior \\   Point\end{tabular} & \begin{tabular}[c]{@{}c@{}}Active\\    set\end{tabular} & $\epsilon$                  & $\beta$ \\ \hline
lp\_brandy   & $1047 \times 303$ & 0.007  & 0.0117 & 16.97                                                       & 63.11                                                   & 0.1 & 50   \\ \hline
lp\_blend    & $337 \times 114$    & 0.41   & 0.56   & 2.28                                                        & 4.62                                                    & 0.001 & 20   \\ \hline
lp\_agg      & $2207 \times 615$   & 0.059  & 0.088  & $66.54^*$                                    & $315.91^*$                               & 0.01 & 50   \\ \hline
lp\_adlittle & $389 \times 138$    & 0.0008 & 0.002  & 2.16                                                        & 4.96                                                    & 0.01 & 10   \\ \hline
lp\_bandm    & $1555 \times 472$   & 0.28   & 0.24   & 14.57                                                       & $529.43^* $                              & 0.01 & 70   \\ \hline
lp\_degen2   & $2403 \times 757$   & 8.29   & 10.16  & 7.13                                                        & 21038                                                   & 0.01 & 200  \\ \hline
lp\_finnis   & $3123 \times 1064$  & 0.13   & 0.15   & $66.16^*$                                    & $237750^*$                               & 0.005                  & 100  \\ \hline
lp\_recipe   & $591 \times 204$    & 0.19   & 0.27   & 0.89                                                        & 63.24                                                   & 0.002                  & 30   \\ \hline
lp\_scorpion & $1709 \times 466$   & 6.83   & 11.86  & 17.68                                                       & 8.02                                                    & 0.005                  & 200  \\ \hline
lp\_stocfor1 & $565 \times 165$    & 0.31   & 0.37   & 2.13                                                        & 2.52                                                    & 0.001                  & 50   \\ \hline
\end{tabular}
\end{center}}
\end{table}

At first, we solve the feasibility problem ($\bar{Ax} \leq \bar{b}$) using SKM and ASKM and recorded the CPU time in Table \ref{tab:1}. But we didn't solve the feasibility problem ($\min 0 \ s.t \ \bar{A}x \leq \bar{b} $) directly using \texttt{fmincon}'s IPM and ASM algorithms since both of these methods fail to solve the feasibility problem due to the fact that in IPM, the \textit{Karush Kuhn Tucker} (KKT) condition system in each iteration becomes singular and similarly ASM halts in the initial step of finding a feasible point.

For fairness of comparison, in Table \ref{tab:1}, we list the CPU time as follows: for SKM and ASKM method we used the feasibility problem ($\bar{Ax} \leq \bar{b}$) and for the \texttt{fmincon} algorithms we used the original optimization LPs ($\min c^Tx \ s.t \ Ax \leq b, \ l \leq x\leq u $). As noted in \cite{haddock:2017}, this is not an obvious comparison. For a better comparison, following \cite{haddock:2017} we set the halting criterion for SKM and ASKM as $\frac{\max(\bar{A}x_k-\bar{b})}{\max(\bar{A}x_0-\bar{b})} \leq \epsilon$ and the halting criterion for the \texttt{fmincon}'s algorithms are set as $\frac{\max(Ax_k-b, l-x_k, x_k-u)}{\max(Ax_0-b, l-x_0, x_0-u)} \leq \epsilon$ and $\frac{c^Tx_k}{c^Tx_0} \leq \epsilon$, where $\epsilon$ is listed in Table \ref{tab:1}. For each problem, every method started with the same initial solution far from the feasible region.

The experiments show that the proposed ASKM method compares favorably with IPM and ASM methods. Notice that the improvement of ASKM over SKM method for some problems are marginal as the analyzed instances contain sparse matrices while our proposed ASKM is explicitly designed for dense problems. Following the method provided by Liu and Wright \cite{wright:2016}, we believe one can develop a sparse version of ASKM algorithm. The trick is to aggregate several steps to reduce the calculation by using the sparsity of the instances. For example, after calculating $x_k, y_k$ and $v_k$, instead of updating to $x_{k+1}, y_{k+1}$ and $v_{k+1}$, for $T \gg 1$ we can update $x_{k+T}, y_{k+T}$ and $v_{k+T}$ using the recurrence relation which will reduce the computational effort significantly.

\section{Conclusion}
\label{sec:colc}
In this work, we have proposed an accelerated version of SKM algorithm for solving LF problem using the celebrated Nesterov acceleration of \textit{Gradient Descent} method. The proposed algorithm also generalizes the accelerated randomized Kaczmarz algorithm for solving LS problems in the context of sample size $\beta$. We have performed a series of numerical experiments to show the performance and effectiveness of our proposed algorithm in comparison with IPM and ASM methods. ASKM algorithm performs favourably in comparison with the original SKM method, IPM and ASM method for a wide range of test instances. The proposed algorithm as it is, including the convergence analysis, can be adopted effectively for both dense and sparse systems, however, we believe, a more efficient algorithm is possible for the sparse case. In the future, we plan to extend this work to solve large-scale real-world problems with greater sparsity on the constraint matrix. Furthermore, due to the introduction of the acceleration to the SKM algorithm, we have a set of parameters (i.e., $\beta, d$ etc.) which we plan to optimize based on the problem structure to further improve the efficiency of the proposed algorithm.

\bibliographystyle{unsrt}
\bibliography{askm}

\begin{thebibliography}{10}

\bibitem{strohmer:2008}
Thomas Strohmer and Roman Vershynin.
\newblock A randomized kaczmarz algorithm with exponential convergence.
\newblock {\em Journal of Fourier Analysis and Applications}, 15(2):262, Apr
  2008.

\bibitem{lewis:2010}
Dennis Leventhal and Adrian~S. Lewis.
\newblock Randomized methods for linear constraints: Convergence rates and
  conditioning.
\newblock {\em Mathematics of Operations Research}, 35(3):641--654, 2010.

\bibitem{needell:2010}
Deanna Needell.
\newblock Randomized kaczmarz solver for noisy linear systems.
\newblock {\em BIT Numerical Mathematics}, 50(2):395--403, Jun 2010.

\bibitem{drineas:2011}
Petros Drineas, Michael~W. Mahoney, Shan Muthukrishnan, and Tam{\'a}s
  Sarl{\'o}s.
\newblock Faster least squares approximation.
\newblock {\em Numerische Mathematik}, 117(2):219--249, Feb 2011.

\bibitem{zouzias:2013}
Anastasios Zouzias and Nikolaos~M. Freris.
\newblock Randomized extended kaczmarz for solving least squares.
\newblock {\em SIAM Journal on Matrix Analysis and Applications},
  34(2):773--793, 2013.

\bibitem{lee:2013}
Yin~Tat Lee and Aaron Sidford.
\newblock Efficient accelerated coordinate descent methods and faster
  algorithms for solving linear systems.
\newblock In {\em Proceedings of the 2013 IEEE 54th Annual Symposium on
  Foundations of Computer Science}, FOCS '13, pages 147--156, Washington, DC,
  USA, 2013. IEEE Computer Society.

\bibitem{ma:2015}
Anna Ma, Deanna Needell, and Aaditya Ramdas.
\newblock Convergence properties of the randomized extended gauss seidel and
  kaczmarz methods.
\newblock {\em {SIAM} Journal on Matrix Analysis and Applications},
  36(4):1590--1604, Jan 2015.

\bibitem{qu:2016}
Zheng Qu, Peter Richtarik, Martin Takac, and Olivier Fercoq.
\newblock {SDNA: Stochastic Dual Newton Ascent for Empirical Risk
  Minimization}.
\newblock In {\em Proceedings of The 33rd International Conference on Machine
  Learning}, volume~48, pages 1823--1832, New York, USA, 20--22 Jun 2016. PMLR.

\bibitem{haddock:2017}
Jes{\'u}s De~Loera, Jamie Haddock, and Deanna Needell.
\newblock A sampling kaczmarz--motzkin algorithm for linear feasibility.
\newblock {\em SIAM Journal on Scientific Computing}, 39(5):S66--S87, 2017.

\bibitem{dembo:1982}
Ron~S. Dembo, Stanley~C. Eisenstat, and Trond Steihaug.
\newblock Inexact newton methods.
\newblock {\em SIAM Journal on Numerical Analysis}, 19(2):400--408, 1982.

\bibitem{stanley:1996}
Stanley~C. Eisenstat and Homer~F. Walker.
\newblock Choosing the forcing terms in an inexact newton method.
\newblock {\em SIAM Journal on Scientific Computing}, 17(1):16--32, 1996.

\bibitem{bellavia:1998}
Stefania Bellavia.
\newblock Inexact interior-point method.
\newblock {\em Journal of Optimization Theory and Applications},
  96(1):109--121, Jan 1998.

\bibitem{xin:2010}
Xin-Yuan Zhao, Defeng Sun, and Kim-Chuan Toh.
\newblock A newton-cg augmented lagrangian method for semidefinite programming.
\newblock {\em SIAM Journal on Optimization}, 20(4):1737--1765, 2010.

\bibitem{wang:2013}
Chengjing Wang and Aimin Xu.
\newblock An inexact accelerated proximal gradient method and a dual newton-cg
  method for the maximal entropy problem.
\newblock {\em Journal of Optimization Theory and Applications},
  157(2):436--450, May 2013.

\bibitem{kaifeng:2012}
Kaifeng Jiang, Defeng Sun, and Kim-Chuan Toh.
\newblock An inexact accelerated proximal gradient method for large scale
  linearly constrained convex sdp.
\newblock {\em SIAM Journal on Optimization}, 22(3):1042--1064, 2012.

\bibitem{jacek:2013}
Jacek Gondzio.
\newblock Convergence analysis of an inexact feasible interior point method for
  convex quadratic programming.
\newblock {\em SIAM Journal on Optimization}, 23(3):1510--1527, 2013.

\bibitem{kaczmarz:1937}
Stefan Kaczmarz.
\newblock Angenaherte auflsung von systemen linearer gleichungen.
\newblock {\em Bulletin International de l'Acadmie Polonaise des Sciences et
  des Letters}, 35:355--357, 1937.

\bibitem{herman:2009}
Gabor~T. Herman.
\newblock {\em Fundamentals of Computerized Tomography: Image Reconstruction
  from Projections}.
\newblock Springer Publishing Company, Incorporated, 2nd edition, 2009.

\bibitem{censor:1981}
Yair Censor.
\newblock Row-action methods for huge and sparse systems and their
  applications.
\newblock {\em SIAM Review}, 23(4):444--466, 1981.

\bibitem{wright:2016}
Ji~Liu and Stephen~J. Wright.
\newblock An accelerated randomized kaczmarz algorithm.
\newblock {\em Math. Comput.}, 85(297):153--178, 2016.

\bibitem{agamon}
Shmuel Agamon.
\newblock The relaxation method for linear inequalities.
\newblock {\em Canadian J. Math}, pages 382--392, 1954.

\bibitem{motzkin}
Theodore~S. Motzkin and Issac~J. Schoenberg.
\newblock The relaxation method for linear inequalities.
\newblock {\em Canadian J. Math}, pages 393--404, 1954.

\bibitem{rosenblatt}
Frank Rosenblatt.
\newblock The perceptron: A probabilistic model for information storage and
  organization in the brain.
\newblock {\em Psychological Review}, pages 65--386, 1958.

\bibitem{ramdas:2014}
Aaditya Ramdas and Javier Peña.
\newblock Margins, kernels and non-linear smoothed perceptrons.
\newblock In {\em Proceedings of the 31st International Conference on Machine
  Learning}, volume~32, pages 244--252, Bejing, China, 22--24 Jun 2014. PMLR.

\bibitem{ramdas:2016}
Aaditya Ramdas and Javier Pe\~{n}a.
\newblock Towards a deeper geometric, analytic and algorithmic understanding of
  margins.
\newblock {\em Optimization Methods and Software}, 31(2):377--391, 2016.

\bibitem{nutini:2016}
Julie Nutini, Behrooz Sepehry, Issam Laradji, Mark Schmidt, Hoyt Koepke, and
  Alim Virani.
\newblock Convergence rates for greedy kaczmarz algorithms, and faster
  randomized kaczmarz rules using the orthogonality graph.
\newblock In {\em Proceedings of the Thirty-Second Conference on Uncertainty in
  Artificial Intelligence}, UAI'16, pages 547--556, Arlington, Virginia, United
  States, 2016. AUAI Press.

\bibitem{petra:2016}
Stefania Petra and Constantin Popa.
\newblock Single projection kaczmarz extended algorithms.
\newblock {\em Numerical Algorithms}, 73(3):791--806, Nov 2016.

\bibitem{hoffman}
Alan~J Hoffman.
\newblock On approximate solutions of systems of linear inequalities.
\newblock In {\em Selected Papers Of Alan J Hoffman: With Commentary}, pages
  174--176. World Scientific, 2003.

\bibitem{xu:2017}
Xu~Xiang, Xu~Liu, Wentang Tan, and Xiang Dai.
\newblock An accelerated randomized extended kaczmarz algorithm.
\newblock {\em Journal of Physics: Conference Series}, 814(1):012017, 2017.

\bibitem{morshed:2018}
Md~Sarowar Morshed and Md. Noor-E-Alam.
\newblock Generalized affine scaling algorithms for linear programming
  problems.
\newblock {\em Computers \& Operations Research}, 114:104807, 2020.

\bibitem{nesterov:1983}
Yuri Nesterov.
\newblock A method for solving the convex programming problem with convergence
  rate $o(1/k\sp{2})$.
\newblock {\em Soviet Mathematics Doklady}, Vol. 27:p(372--376), 1983.

\bibitem{nesterov:2013}
Yuri Nesterov.
\newblock Gradient methods for minimizing composite functions.
\newblock {\em Mathematical Programming}, 140(1):125--161, Aug 2013.

\bibitem{nesterov:2014}
Yuri Nesterov.
\newblock {\em Introductory Lectures on Convex Optimization: A Basic Course}.
\newblock Springer Publishing Company, Incorporated, 1 edition, 2014.

\bibitem{nesterov:2005}
Yuri Nesterov.
\newblock Smooth minimization of non-smooth functions.
\newblock {\em Mathematical Programming}, 103(1):127--152, May 2005.

\bibitem{nesterov:2012}
Yuri Nesterov.
\newblock Efficiency of coordinate descent methods on huge-scale optimization
  problems.
\newblock {\em SIAM Journal on Optimization}, 22(2):341--362, 2012.

\bibitem{gordon:1970}
Richard Gordon, Robert Bender, and Gabor~T. Herman.
\newblock Algebraic reconstruction techniques (art) for three-dimensional
  electron microscopy and x-ray photography.
\newblock {\em Journal of Theoretical Biology}, 29(3):471 -- 481, 1970.

\bibitem{calafiore:2014}
Giuseppe Calafiore and Laurent {El Ghaoui}.
\newblock {\em Optimization Models}.
\newblock Control systems and optimization series. Cambridge University Press,
  October 2014.

\bibitem{lichman:2013}
Moshe Lichman.
\newblock {UCI} machine learning repository, 2013.

\bibitem{yeh:2009}
I-Cheng Yeh and Che-hui Lien.
\newblock The comparisons of data mining techniques for the predictive accuracy
  of probability of default of credit card clients.
\newblock {\em Expert Syst. Appl.}, 36(2):2473--2480, Mar 2009.

\end{thebibliography}

\end{document}